\documentclass{amsart}

\usepackage{amscd,amsmath,amssymb,amsfonts,bbm, calligra, xspace}
\usepackage[T1]{fontenc}
\usepackage{lmodern}
\usepackage{mathtools}
\usepackage{enumerate}
\usepackage{multicol}
\usepackage[all]{xy}
\usepackage{mathrsfs}
\usepackage{footmisc}
\usepackage[unicode,pdfborder={0 0 0},final]{hyperref}

\newtheorem{thm}{Theorem}[section]
\newtheorem{prop}[thm]{Proposition}

\newtheorem{lem}[thm]{Lemma}

\theoremstyle{definition}

\theoremstyle{remark}
\newtheorem{rem}[thm]{Remark}

\newtheorem{Step}{Step}

\numberwithin{equation}{section}

\newcommand{\Sym}{\mathrm{Sym}}
\newcommand{\Spec}{\mathrm{Spec}}
\newcommand{\Sper}{\mathrm{Sper}}
\newcommand{\Frac}{\mathrm{Frac}}

\newcommand{\h}{\mathrm{h}}

\makeatletter
\def\myrightarrow{{\setbox\z@\hbox{$\rightarrow$}\dimen0\ht\z@\multiply\dimen0 6\divide\dimen0 10\ht\z@\dimen0\box\z@}}
\def\myrightarrowfill@{\arrowfill@\relbar\relbar\myrightarrow}
\newcommand{\myxrightarrow}[2][]{\ext@arrow 0359\myrightarrowfill@{#1}{#2}}
\makeatother

\def\loccit{\emph{loc}.\kern3pt \emph{cit}.{}\xspace}
\def\eg{e.g.\kern.3em}
\def\ie{i.e.,\ }
\def\resp {\text{resp.}\kern.3em}

\def\cL{\mathcal{L}}
\def\cO{\mathcal{O}}

\def\cI{\mathcal{I}}

\def\hx{\hat{x}}
\def\hy{\hat{y}}
\def\hz{\hat{z}}

\def\hcO{\widehat{\mathcal{O}}}

\def\kp{\mathfrak{p}}

\def\km{\mathfrak{m}}

\def\ka{\mathfrak{a}}

\def\tx{\tilde{x}}

\def\oZ{\overline{Z}}
\def\ok{\overline{k}}

\def\oD{\overline{D}}

\def\oX{\overline{X}}

\def\oW{\overline{W}}

\def\wD{\widetilde{D}}

\def\whA{\widehat{A}}
\def\whI{\widehat{I}}

\def\A{\mathbb A}

\def\C{\mathbb C}

\def\Q{\mathbb Q}
\def\P{\mathbb P}
\def\R{\mathbb R}

\begin{document}

\title[On the bad points of positive semidefinite polynomials]{On the bad points of positive semidefinite polynomials}

\author{Olivier Benoist}
\address{D\'epartement de math\'ematiques et applications, \'Ecole normale sup\'erieure, CNRS,
45 rue d'Ulm, 75230 Paris Cedex 05, France}
\email{olivier.benoist@ens.fr}

\renewcommand{\abstractname}{Abstract}
\begin{abstract}
A bad point of a positive semidefinite real polynomial $f$ is a point at which a pole appears in all expressions of $f$ as a sum of squares of rational functions. We show that quartic polynomials in three variables never have bad points. We give examples of positive semidefinite polynomials with a bad point at the origin, that are nevertheless sums of squares of formal power series, answering a question of Brumfiel. We also give an example of a positive semidefinite polynomial in three variables with a complex bad point that is not real, answering a question of Scheiderer.
\end{abstract}
\maketitle

\begin{center}
{\textit{Dedicated to Olivier Debarre}}
\end{center}

\section*{Introduction}
\label{intro}

Let $f\in\R[x_1,\dots,x_n]$ be a polynomial with real coefficients that is positive semidefinite, \ie that only takes nonnegative values. Its degree $d$ is then even. Sometimes, one may explain the positivity of $f$ by writing it as a sum of squares of polynomials. Such is the case when $n\leq 1$, when $d\leq 2$, and, as Hilbert proved in \cite{Hilbert88}, when $(n,d)=(2,4)$. 
 For all other values of $(n,d)$, there exist positive semidefinite polynomials that are not sums of squares of polynomials \cite{Hilbert88}.

Hilbert asked in his celebrated $17$th problem whether all positive semidefinite polynomials $f$ could however be written as sums of squares of rational functions.
This was proven by him \cite{Hilbert1893} when $n=2$ and by Artin \cite[Satz 4]{Artin} in general.

To understand the possible denominators in a representation of $f$ as a sum of squares in $\R(x_1,\dots,x_n)$, it is natural to introduce the set $B(f)\subset \C^n$ of \textit{bad points} of~$f$: those points at which some denominator vanishes in all possible representations of $f$ as a sum of squares in $\R(x_1,\dots,x_n)$. The existence of a bad point may be thought of as an explanation why $f$ cannot be a sum of squares in $\R[x_1,\dots,x_n]$.

As indicated in \cite[p.\,20]{Delzellthesis}, that bad points may exist was first noted by Straus in a 1956 letter to Kreisel: if $f\in\R[x_1,\dots,x_n]$ is not a sum of squares of polynomials, then its homogenization in $\R[x_1,\dots,x_{n+1}]$ has a bad point at the origin. 
%The question of the existence of bad points is also raised in [Swan, Topological examples of projective modules, p. 228].
Such examples only appeared in print twenty years later (see \cite[Theorem~4.3]{CL},  \cite[p.\,196]{Brumfiel}, \cite[Proposition 3.5]{CLRR}, \cite[Counterexample~9.1]{BE} or \cite[pp.\,59-61]{Delzellthesis}).

The bad locus $B(f)\subset\C^n$ of $f$ always has codimension $\geq 3$, as was shown in increasing generality by Choi and Lam \cite[Theorem 4.2]{CL}, by
Delzell \cite[Proposition 5.1]{Delzellthesis}, and by Scheiderer \cite[Theorem 4.8]{Schlocal}.
In particular, bad points never appear when $n=2$ (which yields examples of polynomials $f$ with no bad points that are nevertheless not sums of squares of polynomials). 

Our first theorem shows that a similar phenomenon occurs when $(n,d)=(3,4)$.
\begin{thm}[Theorem \ref{44}] 
\label{th1}
Positive semidefinite real polynomials of degree four in three variables have no bad points.
\end{thm}

The $(n,d)=(3,4)$ case considered in this theorem is the only one for which the question of the existence of bad points is not covered by the above-mentioned results. It was ostensibly left open in \cite[Theorem 4.3]{CL}.

Our proof builds on the works of several authors: Hilbert's classical theorem on quartics in two variables \cite{Hilbert88}, Choi, Lam and Reznick's detailed study of quartics in three variables \cite{CLR}, and Scheiderer's general results on sums of squares in local rings \cite{Schlocal}. The argument works over an arbitrary real closed field. 

\vspace{1em}

In three variables, all known examples of bad points share striking common features. To begin with, they are all real points. It was asked by Scheiderer \cite[Remark 1.4~2]{Schregular} whether a positive semidefinite $f\in\R[x,y,z]$ could have a nonreal 
bad point. In our second main theorem, we construct such an example.

\begin{thm}[Theorem \ref{cxbad}]
\label{th2}
There exists a positive semidefinite polynomial in $\R[x,y,z]$ with a bad point that is not real. 
\end{thm}

The only bad points of our example are $(0,0,i)$ and $(0,0,-i)$ (see Theorem \ref{cxbad}).

In $\geq 4$ variables, examples of nonreal bad points were already known since the bad locus $B(f)$ may have dimension $\geq 1$ (see \cite[Example~1~p.\,59]{Delzellthesis}). However, Theorem~\ref{th2} is the first example in any number of variables where the real bad points of $f$ are not Zariski-dense in~$B(f)$.

\vspace{1em}

Additionally, in all existing examples of positive semidefinite $f\in\R[x,y,z]$ with a real bad point, assumed to be the origin, this point is shown to be bad by an analysis of some low degree monomials of $f$. As a consequence, the polynomial~$f$ is not even a sum of squares in the ring $\R[[x,y,z]]$ of formal power series. 
An old question of Brumfiel appearing in \cite[p.\,62]{Delzellthesis} asks whether this is a general phenomenon. In our third main theorem, we answer this question in the negative.

\begin{thm}[Theorem \ref{realbad}]
\label{th3}
There exists a positive semidefinite polynomial in $\R[x,y,z]$ that has a bad point at the origin, but that is a sum of squares in $\R[[x,y,z]]$.
\end{thm}
%Trouver les/des (n,d) où ce phénomène n'apparaît pas? Au vu de Schlocal, ça inclut les (n,d) tq infinité de points réels implique somme de carrés. Mais il n'y en a pas beaucoup, cf [Choi Lam ,Old, pp. 397-398].

Our example does not have other bad points than the origin (see Theorem \ref{realbad}).

Brumfiel asked his question in any number of variables. There are however easier examples in $\geq 4$ variables, as it may happen that  a positive semidefinite $f\in \R[w,x,y,z]$ is a sum of squares in $\R[[w,x,y,z]]$ but not in $\R[w,x,y,z]_{\langle w,x,y,z\rangle}$ for the simple reason that it is not even a sum of squares in some other completion of $\R[w,x,y,z]_{\langle w,x,y,z\rangle}$. We give such an example in Theorem \ref{Motzkin4}.

\vspace{1em}

This last remark points to what is difficult in proving Theorems \ref{th2} and~\ref{th3}. Let $\km\subset\R[x,y,z]$ be the maximal ideal corresponding to the bad point. Under the hypotheses of either theorem, the polynomial $f$ has to be a sum of squares in all the completions of $\R[x,y,z]_{\km}$ (apply \cite[Corollary 2.4 and Theorem 4.8]{Schlocal}). We thus need to devise an obstruction to~$f$ being a sum of squares in the local ring $\R[x,y,z]_{\km}$ that is sufficiently global in nature to allow $f$ to be a sum of squares in all the completions of $\R[x,y,z]_{\km}$. We now briefly explain how to overcome this difficulty (see Section \ref{sec2} for more details).

Let $\Gamma\subset\A^3_{\R}:=\Spec(\R[x,y,z])$ be an integral curve through $\km$ whose real locus $\Gamma(\R)$ is Zariski-dense in $\Gamma$ and such that $f$ vanishes on $\Gamma$. It follows from these facts that, in any representation $f=\sum_i f_i^2$ of $f$ as a sum of squares in $\R[x,y,z]_{\km}$, the $f_i$ must vanish on $\Gamma$. As a consequence, one has $f\in (I_{\Gamma}^2)_{\km}$, where $I_{\Gamma}$ is the ideal defining $\Gamma$. 
It thus suffices to arrange that $f\notin (I_{\Gamma}^2)_{\km}$ to ensure that it is not a sum of squares in $\R[x,y,z]_{\km}$. 

This is not easy to achieve. Indeed, since $f$ is positive semidefinite, it belongs to the ideal $I_{\Gamma}^2$ at all smooth real points of $\Gamma$, hence generically along $\Gamma$. In other words, it belongs to the symbolic square $I_{\Gamma}^{(2)}$ of $I_{\Gamma}$ (see (\ref{symbolicdef}) for the definition of~$I_{\Gamma}^{(2)}$ and the survey \cite{symbolic} for more information on this topic).
 We thus need
the ideals 
$I_{\Gamma}^2$ and~$I_{\Gamma}^{(2)}$ to be distinct.
 The simplest example of this phenomenon, already appearing in \cite[Example 3~p.\,29]{Northcott}, is the ideal of the image of 
 the morphism
 $t\mapsto (t^3,t^4,t^5)$.

The polynomials we use to prove Theorems \ref{th2} and~\ref{th3} are both constructed by modifying appropriately this basic example. For Theorem \ref{th2}, this strategy leads to the concrete polynomial of degree ten $x^{10}+x^2y^6+(z^2+1)^3-3x^4y^2(z^2+1)$ (see Theorem~\ref{cxbad}). The proof of Theorem~\ref{th3} is more involved and does not yield an explicit example.

\vspace{1em}

Our strategy actually works on arbitrary smooth varieties over any base field. We thus obtain the following result. Recall that a field is said to be \textit{formally real} if it admits a field ordering (in particular, such a field has characteristic $0$).

\begin{thm}[Theorem \ref{main}]
\label{th4}
Let $X$ be an affine variety over a field $k$. Let $A$ be a local ring of $X$ that is regular, with maximal ideal $\km$. 
Assume that $\dim(A)\geq 3$ and that $\Frac(A)$ is formally real.
Then there exists $f\in\cO(X)$ such that:
\begin{enumerate}[(i)]
\item The element $f$ is a sum of squares in the completion $\whA_{\km}$ of $A$ at $\km$.
\item For all prime ideals $\kp\neq \km$ of $A$, $f$ is a sum of squares in the localization $A_{\kp}$.
\item But $f$ is not  a sum of squares in $A$.
\end{enumerate}
\end{thm}

Notice that Theorem \ref{th2} (\resp Theorem \ref{th3}) may be obtained as the particular case of Theorem \ref{th4} where $k=\R$, $X=\A^3_{\R}$ and $\km$ has residue field $\C$ (\resp $\R$).

Theorem \ref{th4} yields the first examples of a regular local ring $A$ with $2\in A^*$
%in char 2, see the example at the end of the answer of darx to MO338635.
and of an element $f\in A$ that is a sum of squares in all the completions of $A$ but not in $A$. Such examples do not exist if $\dim(A)\leq 2$ by \cite[Theorem 4.8]{Schlocal}, or if $\Frac(A)$ is not formally real \cite[Corollaries 1.5 and~2.4]{Schlocal}.

Thanks to Theorem \ref{th4}, we are able to complete the proof of the following result, which is almost entirely due to Scheiderer (the case that was still open is explicitly mentioned in \cite[Remark 1.4~2]{Schregular}). To state it, we recall that an element $f$ of a ring~$A$ is \textit{positive semidefinite} if it is nonnegative with respect to all the orderings of the residue fields of $A$.

\begin{thm}
\label{th5}
Let $A$ be the local ring at a regular point of a variety over a field $k$ of characteristic not $2$.
%in characteristic $2$, sums of squares are squares.
The following are equivalent:
\begin{enumerate}[(i)]
\item All positive semidefinite elements of $A$ are sums of squares in $A$.
\item Either $\dim(A)\leq 2$ or $\Frac(A)$ is not formally real.
\end{enumerate}
\end{thm}

\begin{proof}
If $\dim(A)\leq 2$, one may apply \cite[Theorem 4.8]{Schlocal}, and if $\Frac(A)$ is not formally real, the theorem follows from \cite[Corollaries 1.5 and~2.4]{Schlocal}. The other cases are covered by Theorem \ref{th4}, but were already known if either $\dim(A)\geq 4$ or if the residue field of $A$ is formally real (see \cite[Propositions 1.2 and 1.5]{Schregular}).
\end{proof}

Understanding when assertion (i) of Theorem \ref{th5} holds is also interesting when~$A$ is possibly singular. We refer to \cite[Theorem 3.9]{Schlocal}, to \cite[Theorem 3.1]{Fernando2} 
%classification of real analytic surface germs of embedding dimension 3 with psd=sos.
%See [Scheiderer, Weighted sums of squares 1, Remarks 3.26 and 4.8] for extensions.
and to \cite[Theorem 1.1]{FRS} for the best known results in dimensions $1$, $2$ and $\geq 3$ respectively.
%Ne sais-je pas toujours traiter d>=3? Difficulté pour singularités non Cohen-Macaulay car on ne peut se ramener à un exemple lisse, non ?

It is tempting to ask if Theorem \ref{th5} remains true for arbitrary regular local rings, not necessarily of geometric origin. In our last result, we show that this is not the case, answering a question raised in \cite[bottom of p.\,209]{Schlocal}.

\begin{thm}[Theorem \ref{regloc}]
\label{th6}
For all $n\geq 0$, there exists a regular local $\R$\nobreakdash-algebra~$A$ of dimension $n$ with the following properties:
\begin{enumerate}[(i)]
\item All positive semidefinite elements of $A$ are sums of squares in $A$.
\item The field $\Frac(A)$ is formally real.
\end{enumerate}
\end{thm}

The regular local rings that we consider to prove Theorem \ref{th6}  are actually not far from geometry. When $n\geq 1$, they lie between the local ring of $\A^n_{\R}$ at a closed point with complex residue field and its Henselization.

\subsection*{Notation}

If $A$ is ring and $\kp\subset A$ is a prime ideal, we let $A_{\kp}$ and~$\whA_{\kp}$ be the localization and the completion of $A$ at $\kp$, and we denote by $I_{\kp}=IA_{\kp}\subset A_{\kp}$ and $\whI_{\kp}=I\whA_{\kp}\subset \whA_{\kp}$ the ideals generated by an ideal $I\subset A$.

An algebraic variety $X$ over a field $k$ is a separated scheme of finite type over~$k$. If $k'$ is a field extension of $k$, we denote by $X_{k'}:=X\times_k k'$ the extension of scalars, and by $X(k')$ the set of $k'$-points of $X$. 

\subsection*{Acknowledgements}

I thank Charles N. Delzell and Claus Scheiderer for having made \cite{Delzellthesis} available to me, Karim Johannes Becher for useful comments, and the referee for their careful work.

\section{Generalities on real spectra and sums of squares}
\label{sec0}

The \textit{real spectrum} $\Sper(A)$ of a ring $A$ is the set of pairs $\xi=(\kp,\prec)$, where $\kp$ is a prime ideal of $A$ and $\prec$ is a field ordering of $\Frac(A/\kp)$. The element $\xi\in \Sper(A)$ is said to be \textit{supported} at $\kp$. We denote by $\prec_{\xi}$ the ordering associated with $\xi$. We endow $\Sper(A)$ with its \textit{spectral topology} \cite[Definition 7.1.3]{BCR}, generated by open sets of the form $\{\xi\in\Sper(A) \mid f_i\succ_{\xi} 0\}$ for ${(f_i)_{1\leq i\leq m}\in A^m}$. 
If $\xi,\zeta\in\Sper(A)$, one says that $\xi$ is a \textit{specialization} of $\zeta$ 
if $\xi$ belongs to the closure of $\zeta$.
An element ${f\in A}$ is \textit{positive semidefinite} (\resp \textit{totally positive}) if it is nonnegative (\resp positive) with respect to all points of $\Sper(A)$. A real polynomial $f\in\R[x_1,\dots,x_n]$ is positive semidefinite in this sense if and only if it is positive semidefinite in the sense considered in the introduction (see \cite[Propositions 7.2.1 and 7.2.2]{BCR}).

If $k$ is a field, then $\Sper(k)$ coincides with the set of field orderings of $k$ endowed with the Harrison topology (see \cite[VIII, \S 6]{Lam}). 
The field $k$ is said to be \textit{formally real} if $\Sper(k)$ is nonempty, \ie if $k$ admits a field ordering.

\vspace{1em}

We now collect a few known statements that will be used repeatedly in the sequel. We start with two lemmas.

\begin{lem}[{\cite[Lemma 0.1]{Schregular}}]
\label{dense}
Let $A$ be a regular domain with fraction field~$K$. Then $\Sper(K)$ is dense in $\Sper(A)$.
\end{lem}

\begin{lem}[{\cite[Lemma 5.1\,a)]{Schregular}}]
\label{lterm}
Let $A$ be a regular local ring with maximal ideal~$\km$. View $\km/\km^2$ as an $A/\km$-vector space. Let $f\in A$ be positive semidefinite. If $f\in\km^d$, the image of $f$ in ${\km^{d}/\km^{d+1}=\Sym^d(\km/\km^2)}$ is a positive semidefinite polynomial function~on the dual vector space $(\km/\km^2)^{\vee}$.

In particular, if $A/\km$ is formally real and $d$ is odd, then $f\in\km^{d+1}$.
\end{lem}

The following two theorems are due to Scheiderer.

\begin{thm}[{\cite[Corollary 2.4]{Schlocal}}]
\label{totpos}
Let $A$ be a local ring with $2\in A^*$. If $f\in A$ is totally positive, then $f$ is a sum of squares in $A$.
\end{thm}

\begin{thm}[{\cite[Theorem 4.8]{Schlocal}}]
\label{dim2}
Let $A$ be a regular local ring of dimension two with $2\in A^*$. If $f\in A$ is positive semidefinite, then $f$ is a sum of~squares in $A$.\end{thm}

\section{Quartics in three variables have no bad points}
\label{sec1}

In this section, we show that positive semidefinite quartic polynomials in three variables have no bad points (Theorem \ref{44}). 
We also study quartics in three variables that are nonnegative in a neighbourhood of the origin (see Theorem~\ref{quartic2}).

Throughout, we work over a real closed field $R$. 
We start with a series of lemmas.

\begin{lem}
\label{adic}
Let $B$ be a ring with $2\in B^*$. Fix $g\in B[[x_1,\dots,x_n]]$ with only terms of degree $\geq 3$. Choose $0\leq r\leq n$. Then there exist $(a_i)_{1\leq i \leq r}\in B[[x_1,\dots,x_n]]^r$ with only  terms of degree $\geq 2$, and $b\in B[[x_{r+1},\dots,x_n]]$, such that $$\sum_{i=1}^r x_i^2 +g=\sum_{i=1}^r (x_i+a_i)^2+b.$$
\end{lem}

\begin{proof}
By induction on $N\geq 1$, we will construct  $(a_{i,N})_{1\leq i\leq r}\in B[x_1,\dots, x_n]^r$, $b_N\in B[x_{r+1},\dots, x_n]$ and $c_N\in B[[x_1,\dots, x_n]]$ with the following properties:
\begin{enumerate}[(i)]
\item One has $\sum_{i=1}^r x_i^2 +g=\sum_{i=1}^r (x_i+a_{i,N})^2+b_N+c_N$.
\item  Only terms of degree $\geq 2$ appear in $a_{i,N}$.
\item Only terms of degree $\geq N+2$ appear in $c_N$.
\item Only terms of degree $\geq N+1$ appear in $a_{i,N+1}-a_{i,N}$ and in $b_{N+1}-b_N$.
\end{enumerate}
To do so, we set $a_{i,1}=b_1=0$ and $c_1=g$. If $a_{i,N}$, $b_N$ and $c_N$ have been constructed, write the degree $N+2$ term of $c_N$ as $\sum_{i=1}^r x_iu_i+v$, where $u_i\in B[x_1,\dots,x_n]$ has degree $N+1$ and $v\in B[x_{r+1},\dots,x_n]$ has degree $N+2$. It now suffices to define $a_{i,N+1}=a_{i,N}+u_i/2$, $b_{N+1}=b_N+v$ and $c_{N+1}=c_N-v-\sum_{i=1}^r(x_iu_i+a_{i,N}u_i+u_i^2/4)$.

To conclude, define $a_i:=\lim_{N\to\infty} a_{i,N}$ and $b:=\lim_{N\to\infty}b_N$, where the limits are taken with respect to the $\langle x_1,\dots,x_n\rangle$-adic topology.
\end{proof}

%\begin{lem}
%\label{seriepos}
%The lowest degree term of a positive semidefinite formal power series $f\in R[[x_1,\dots,x_n]]$ is a positive semidefinite polynomial. It has even degree.
%\end{lem}

%\begin{proof}
%Let $d$ be the degree of the lowest degree term of $f$. After a generic linear change of coordinates, we may assume that the monomial $x_1^d$ appears in $f$. Since $f$ is positive semidefinite, so is its image $\of$ in $R[[x_1,\dots,x_n]]/\langle x_i\rangle_{i>1}=R[[x_1]]$. Consider the two orderings of the field $R((x_1))$: the one for which $x_1$ is a positive infinitesimal, and the one for which it is a negative infinitesimal.
%As $\of$ is positive with respect to both, the degree of its lowest degree term, which is $d$, is even.

%Set $A:=R[[x_1,\dots,x_n]]$ and define $B:=A[y_2,\dots, y_n]/\langle y_ix_1-x_i\rangle_{i>1}$. In this ring, one may write $f=gx_1^d$ for some $g\in B$. Since $f$ is positive semidefinite in $A$, hence in $B$, and since $x_1^d$ is a square because $d$ is even, we see that $g\in B$ is positive semidefinite. Consequently, the image $\og$ of $g$ in the quotient $B/\langle x_1\rangle=R[y_2,\dots, y_n]$ is positive semidefinite. Since $\og$ coincides with the dehomogenization of the degree~$d$ term of~$f$ with respect to $x_1$, the lemma is proven.
%\end{proof}

\begin{lem}
\label{quadraterms}
Let $f\in R[[x_1,\dots,x_n]]$ be positive semidefinite. Assume that the lowest degree term of $f$ is a quadratic form of rank $r\geq n-2$. Then $f$ is a sum of squares in $R[[x_1,\dots,x_n]]$.
\end{lem}

\begin{proof}
The degree $2$ term of $f$ is positive semidefinite by Lemma \ref{lterm}. Since it may be diagonalized after a linear change of coordinates, we may assume that ${f=\sum_{i=1}^r x_i^2+g}$, where all monomials in $g\in R[[x_1,\dots,x_n]]$ have degree $\geq 3$. 

By Lemma \ref{adic} applied with $B=R$, there exist $(a_i)_{1\leq i \leq r}\in R[[x_1,\dots,x_n]]^r$ with only terms of degree $\geq 2$, and $b\in R[[x_{r+1},\dots,x_n]]$, such that $f=\sum_{i=1}^r (x_i+a_i)^2+b$ (by convention, $R[[x_{r+1},\dots,x_n]]=R$ when $r=n-2$).

 Since $f$ is positive semidefinite, so is its image in $R[[x_1,\dots,x_n]]/\langle x_i+a_{i}\rangle_{1\leq i\leq r}$, showing that $b\in R[[x_{r+1},\dots,x_n]]$ is positive semidefinite. As $r\geq n-2$, Theorem~\ref{dim2} shows that $b$ is a sum of squares in $R[[x_{r+1},\dots,x_n]]$. The proof is now complete.
\end{proof}

\begin{lem}
\label{lem44}
A polynomial $f\in R[x,y,z]$ of degree $\leq 4$ which is positive semi\-definite in $R[[x,y,z]]$
 is a sum of squares in $R[[x,y,z]]$.
\end{lem}

\begin{proof}
By Lemma \ref{lterm}, the lowest degree term of $f$ is positive semidefinite of degree~$0$, $2$ or $4$. If it has degree $0$, then $f$ is a square in $R[[x,y,z]]$, and we are done. If it has degree $2$, one may apply Lemma \ref{quadraterms} to conclude. If it has degree~$4$, then~$f$ is a sum of three squares of quadratic polynomials by Hilbert's theorem \cite{Hilbert88} (this result, proven over $\R$ in \loccit, holds over an arbitrary real closed field by the Tarski-Seidenberg principle \cite[Proposition 5.2.3]{BCR}).
\end{proof}

Now comes the main theorem of Section \ref{sec1}.

\begin{thm}
\label{44}
Let $f\in R[x,y,z]$ be positive semidefinite of degree $\leq 4$. For all maximal ideals $\km\subset R[x,y,z]$, the polynomial $f$ is a sum of squares in $R[x,y,z]_{\km}$.
\end{thm}
%trois?

\begin{proof}
If $f$ has at least twelve real zeros, then it is a sum of six squares of quadratic polynomials, by a theorem of Choi, Lam and Reznick \cite[Theorem 5.1]{CLR} (this fact, proven over the reals in \loccit , is valid over any real closed field by the Tarski-Seidenberg principle \cite[Proposition 5.2.3]{BCR}). 

We may thus assume that $f$ has finitely many real zeros. Using \cite[Theorem~7.2.3]{BCR}, we see that $\Sper((R[x,y,z]/\langle f\rangle)_{\km})$ contains exactly one point (which is supported at $\km$) if the residue field of $\km$ is $R$, and is empty otherwise. In the latter case, the element $f\in R[x,y,z]_{\km}$ is totally positive, hence a sum of squares by Theorem \ref{totpos}. 
  
It remains to deal with the case where the residue field of $\km$ is $R$. After changing coordinates by a translation,  one may suppose that $\km=\langle x,y,z\rangle$. 
By Lemma~\ref{lem44}, the polynomial $f$ is a sum of squares in $R[[x,y,z]]$.
Since $\Sper((R[x,y,z]/\langle f\rangle)_{\km})$ is supported at $\km$, a theorem of Scheiderer \cite[Corollary 2.7 (ii)$\Rightarrow$(i)]{Schlocal} shows that~$f$ is a sum of squares in $R[x,y,z]_{\km}$, as wanted.
\end{proof}

We conclude this section with a more concrete reformulation of Lemma \ref{lem44}.

\begin{thm}
\label{quartic2}
Let $f\in R[x,y,z]$ be of degree $\leq 4$. The following are equivalent:
\begin{enumerate}[(i)]
\item The function $f:R^3\to R$ takes only nonegative values in a Euclidean neighbourhood of the origin.
\item The polynomial $f$ is a sum of squares in $R[[x,y,z]]$.
\end{enumerate}
\end{thm}

\begin{proof}
It suffices to combine Lemma \ref{lem44} and Lemma \ref{ruiz} below.
\end{proof}

\begin{lem}
\label{ruiz}
Let $A$ be a finitely generated $R$-algebra, fix $f\in A$ and let $\km\subset A$ be a maximal ideal with residue field $R$. The following assertions are equivalent:
\begin{enumerate}[(i)]
\item The element $f\in\whA_{\km}$ is positive semidefinite.
\item The function $\Spec(A)(R)\to R$ induced by $f$ takes only nonnegative values in a Euclidean neighbourhood of the point of $\Spec(A)(R)$ corresponding to $\km$.
\end{enumerate}
\end{lem}

\begin{proof}
Let $x$ (\resp $\tx$) be the point of $\Spec(A)(R)$ (\resp of $\Sper(A)$) associated to~$\km$.
By \cite[Th\'eor\`eme 1.1]{Ruiz}, the image of the natural map $\Sper(\whA_{\km})\to \Sper(A)$ consists exactly of the elements having $\tx$ as a specialization.
%more general result in "A Going-Down Theorem for Real Spectra"
%That all points of \Sper(\whA_{\km}) specialize to \km is asserted at the very beginning of Ruiz, under Henselian hypothesis.

If (ii) holds, the semi-algebraic subset $\{f\geq 0\}$ of $\Spec(A)(R)$ contains a neighbourhood of $x$. It follows from \cite[Theorem 7.2.3]{BCR} that the constructible subset $\{f\geq 0\}$ of $\Sper(A)$ contains a neighbourhood of $\tx$, hence all points having~$\tx$ as a specialization. Consequently, $f$ is positive semidefinite in $\whA_{\km}$.

Conversely, assume that (ii) does not hold, hence that the open subset $\{f<0\}$ of $\Spec(A)(R)$ contains $x$ in its closure. By \cite[Theorem 7.2.3]{BCR}, the open subset $\{f<0\}$ of $\Sper(A)$ contains $\tx$ in its closure. In view of \cite[Proposition~7.1.21]{BCR}, this subset of $\Sper(A)$ contains a point specializing to $\tx$. This shows that $f$ is not positive semidefinite as an element of $\whA_{\km}$.
\end{proof}

\begin{rem}
Scheiderer has shown in \cite[Theorem 2.1]{Schrat} the existence of a homogeneous polynomial $f\in\Q[x,y,z]$ of degree $4$ that is positive semidefinite, but not a sum of squares in $\Q[x,y,z]$. Using the homogeneity of $f$, one sees that $f$ is not a sum of squares in $\Q[[x,y,z]]$ either. This shows that Theorems \ref{44} and \ref{quartic2} cannot be extended to general base fields that are not necessarily real closed.
\end{rem}

\section{Examples of bad points}
\label{sec2}

In \S \ref{criterion}, we state a simple criterion for an element of a ring not to be a sum of squares.
This criterion is applied in \S\ref{nonreal} and \S\ref{formal} to give examples of real positive semidefinite polynomials in three variables with a nonreal bad point, or with a real bad point that cannot be detected after completion (Theorems \ref{cxbad} and~\ref{realbad}). We apply it again in \S\ref{general} to give examples of regular bad points on varieties over a field that satisfy minimal hypotheses
(Theorem \ref{main}). Another example of bad point, of a different nature, is presented in~\S\ref{example}.

The proofs of Theorems \ref{cxbad}, \ref{realbad} and \ref{main} rely on auxiliary polynomials $f_1$, $f_2$, $f_3$ and $f_4$, respectively constructed in Lemma \ref{symb}, Proposition \ref{cdvariables}, Proposition~\ref{lempos} and Proposition~\ref{lempos2}. For $1\leq i\leq 3$, the polynomial $f_i$ is used to construct~$f_{i+1}$. While the polynomial $f_1$ is simple (see Lemma \ref{symb}), the expression of $f_2$ is quite complicated (see Remark \ref{remf2}), and we do not provide explicit formulas for $f_3$ and~$f_4$.

In this whole section, we let $k$ be a field of characteristic $0$.

\subsection{A criterion to be a bad point}
\label{criterion}

We will use the following easy lemma.

\begin{lem}
\label{lemcrit}
Let $I$ be a radical ideal in a ring $A$
such that the image of $\Sper(A/I)$ in $\Spec(A/I)$ is Zariski-dense. If $f\in I\setminus I^2$, then $f$ is not a sum of squares in $A$.
\end{lem}

\begin{proof}
Assume for contradiction that $f=\sum_i f_i^2$ is a sum of squares in $A$. Then $\sum_i f_i^2$ vanishes in $A/I$. It follows that the $f_i$ vanish at all formally real residue fields of $A/I$. As $\Sper(A/I)$ is Zariski-dense in $\Spec(A/I)$ and as $A/I$ is reduced, the $f_i$ vanish in $A/I$. The $f_i$ thus belong to $I$, so that $f\in I^2$, which is absurd.
\end{proof}

Recall that the symbolic square of an ideal $I$ in a Noetherian ring $A$~is
\begin{equation}
\label{symbolicdef}
I^{(2)}:=\{f\in A \mid f\in I^2_\kp\textrm{ for all associated primes }\kp\textrm{ of }A/I\}.
\end{equation}
The following lemma will not be used in the sequel, but explains why it may be difficult to apply Lemma \ref{lemcrit} in practice.

\begin{lem}
\label{difficulty}
Under the hypotheses of Lemma \ref{lemcrit}, if the ring $A$ is regular and the element $f\in A$ is positive semidefinite, then $f\in I^{(2)}$.
\end{lem}

\begin{proof}
Let $\kp$ be an associated prime ideal of $A/I$.
Then $\kp A_{\kp}=I_{\kp}$ because $I$ is radical.
As $\Sper(A/I)$ is Zariski-dense in $\Spec(A/I)$, one cannot write $-1$ as a sum of squares in $\kappa:=\Frac(A/\kp)$, so that $\kappa$ is formally real by \cite[Theorem 1.1.8]{BCR}. Since $f\in\kp A_{\kp}$ is positive semidefinite and $\kappa$ is formally real, Lemma \ref{lterm} shows that $f\in\kp^2A_{\kp}=I_{\kp}^2$.
\end{proof}

Consequently, to apply Lemma \ref{lemcrit} to give an example of a positive semidefinite element in a regular ring that is not a sum of squares, we must ensure that $I^{(2)}\neq I^2$. A basic example of an ideal in a regular ring whose square and symbolic square are distinct will be given later, in Lemma \ref{symb}.

\subsection{A criterion to be a sum of squares}
\label{parsum}

In our proofs of Theorems \ref{cxbad} and~\ref{realbad}, we also need a way to check that a regular function on a variety over $k$ is a sum of squares in a neighbourhood of a point.
This is the role of Proposition \ref{glob}.

\begin{prop}
\label{glob}
Let $X$ be a smooth affine variety of dimension $n$ over $k$. Let $f\in\cO(X)$ be positive semidefinite, and let $p\in X$ be a closed point. 
Let $Y\subset X$ be the Zariski closure of $\Sper(\cO(X)/\langle f\rangle)$. If $p\in Y$, then the differential of $f$ at $p$ vanishes. If moreover $Y$ is smooth of dimension~$n-c$ at $p$ and the Hessian of $f$ at~$p$ has rank $\geq c$, then $f$ is a sum of squares~in~$\cO_{X,p}$.
\end{prop}

\begin{proof}
By Lemma \ref{lterm}, the differential of $f$ vanishes at all points $x\in X$ with formally real residue field such that $f(x)=0$. It follows that the differential of $f$ vanishes on $Y$, hence at $p$.

Set $A:=\cO_{X,p}$ and let $I\subset A$ be the ideal of functions vanishing on $\Sper(A/\langle f\rangle)$. 
As the generic points of $Y$ are formally real by Lemma \ref{dense}, we see that 
 $I$ is the ideal of the subscheme $Y\times_X \Spec(A)$ of $\Spec(A)$. 
%\cite[Theorem~7.2.3]{BCR}. 

Since $Y$ is smooth of dimension $n-c$ at $p$, the local ring $B:=A/I$ is regular of dimension $n-c$. 
We have seen above that the fraction field of $B$ is formally real. As $f$ is positive semidefinite, it follows from Lemma \ref{lterm} that the image of $f$ in the localization $A_I$ belongs to $I^2A_I$. We deduce from \cite[Appendix 6, Lemma 5]{ZS2} 
applied with $\ka=I$ (or from \cite[(2.1)]{Hochster}) that $f\in I^2$.

As~$Y$ is smooth at $p$, \cite[Theorem~30.3~(1)$\Rightarrow$(2)]{Matsumura} 
%%Uses essentially of finite type very much, see [Matsumura, Remark p. 215].
shows that~$B$ is $0$\nobreakdash-smooth over~$k$ in the sense of \cite[p.~193]{Matsumura}, hence that the natural surjections $A/I^n\to B$ admit compatible sections for $n\geq 1$. This yields a section $s:B\to\whA_{I}$ of the quotient map $\whA_{I}\to\whA_I/I=B$. Let $x_1,\dots, x_c$ be generators of $I$. 
By \cite[Theorem~7.16]{Eisenbud}, the section $s$ induces a surjective morphism of $B$\nobreakdash-algebras $B[[x_1,\dots,x_c]]\to \whA_{I}$. As $\whA_I$ is faithfully flat over $A$ by \cite[Theorem 8.14]{Matsumura}, one may apply \cite[Theorem 15.1 (ii)]{Matsumura} to show that $\dim(\whA_I)\geq\dim(A)=n$. Since $B[[x_1,\dots,x_c]]$ is integral of dimension $n$, the surjection $B[[x_1,\dots,x_c]]\to \whA_{I}$ is then necessarily an isomorphism.

We now argue as in the proof of \cite[Corollary~2.7]{Schlocal}.
By a theorem of Scheiderer \cite[Theorem~2.5]{Schlocal}, there exists an ideal $J$ of $A$ with radical $I$ such that~$f$ is a sum of squares in~$A$ if it is a sum of squares in $A/J$. Since $I^m\subset J$ for some $m$, the proposition will be proven if we check that~$f$ is a sum of squares in $A/I^m$ for all $m$. We will show the stronger fact that $f$ is a sum of squares in $\whA_I=B[[x_1,\dots,x_c]]$. We have seen above that $f\in I^2=\langle x_1,\dots,x_c\rangle^2$. Let $\kappa$ be the residue field of~$B$. Since the Hessian of $f$ has rank $\geq c$, the image in $\kappa[x_1,\dots,x_c]$ of the quadratic term of $f$ is a nondegenerate quadratic form. Applying Lemma~\ref{loc} below concludes.
\end{proof}

\begin{lem}
\label{loc}
Let $B$ be a local ring whose residue field $\kappa$ is not of characteristic~$2$. Let ${f\in B[[x_1,\dots,x_n]]}$ be positive semidefinite. If the lowest degree term of $f$ is quadratic with nondegenerate image in
$\kappa[x_1,\dots,x_n]$, then $f$ is a sum of squares. 
\end{lem}

\begin{proof}
Let $h\in B[x_1,\dots,x_n]$ be the quadratic term of $f$. By 
\cite[Chapter I, Proposition 3.4]{Baeza}, we may assume after a suitable linear change of coordinates that $h=\sum_{i=1}^n\alpha_ix_i^2$ for some invertible elements $\alpha_i\in B$. Since $f$ is positive semidefinite, the $\alpha_i$ are positive semidefinite, and Theorem \ref{totpos} allows us to write $\alpha_i=\sum_j(\alpha_i^{(j)})^2$ for some $\alpha_i^{(j)}\in B$. After maybe permuting the $\alpha_i^{(j)}$, we may ensure that $\alpha_i^{(1)}$ is invertible in~$B$. Choosing the $\alpha_i^{(1)}x_i$ as new variables, we may finally assume that $h-\sum_{i=1}^n x_i^2$ is a sum of squares in $B[x_1,\dots, x_n]$.
By Lemma \ref{adic} applied with ${g=f-h}$, there exist $a_1,\dots,a_n$ in $B[[x_1,\dots,x_n]]$ with 
$$\sum_{i=1}^n x_i^2+f-h=\sum_{i=1}^n(x_i+a_i)^2.$$
Combining these two facts shows that $f$ is a sum of squares in $B[[x_1,\dots,x_n]]$.
\end{proof}

%If $A$ is a regular local ring with maximal ideal $\km$ and $f\in \km^2$, the \textit{Hessian} of~$f$ is the quadratic form on the $A/\km$-vector space $(\km/\km^2)^{\vee}$ defined by $f\in\km^2/\km^3$.

%\begin{prop}
%\label{glob}
%Let $A$ be a regular local $\Q$-algebra of dimension $n$. Let $f\in A$ be positive semidefinite.
%Let $I\subset A$ be the ideal of functions vanishing on $\Sper(A/\langle f\rangle)$. If $B:=A/I$ is regular of dimension $n-c\geq 0$, then $f\in\km^2$. If moreover the Hessian of $f$ has rang $\geq c$, then $f$ is a sum of squares in $A$.
%\end{prop}

%\begin{proof}
%By \cite[Lemma 0.1]{Schregular}, the fraction field of $B$ is formally real. As $f$ is positive semidefinite, it follows from \cite[Lemma~5.1]{Schregular} that the image of $f$ in the localization $A_I$ belongs to $I^2A_I$. We deduce from \cite[Appendix 6, Lemma 5]{ZS2} 
%applied with $\ka=I$ (or from \cite[(2.1)]{Hochster}) that $f\in I^2$. In particular, $f\in\km^2$.
%\end{proof}

\subsection{A nonreal bad point}
\label{nonreal}

As explained in \S\ref{criterion}, we are in need of an ideal whose square and symbolic square differ. Lemma \ref{symb} contains a simple~example.

\begin{lem}
\label{symb}
Let $C\subset\A^3_k$ be the image of the morphism $\nu:\A^1_k\to\A^3_k$ given by $\nu(t)=(t^3,t^4,t^5)$. Define $I_C:=\langle u^3-vw,v^2-uw, w^2-u^2v\rangle\subset k[u,v,w]$ and consider the polynomial $f_1:=u^5+uv^3+w^3-3u^2vw$.
The following properties~hold: 
\begin{enumerate}[(i)]
\item The zero locus of $I_C$ is the geometrically integral curve $C\subset \A^3_{k}$.
\item One has $f_1\in I_C$, $f_1\notin I_{C,\langle u,v,w\rangle}^2$ and $vf_1\in I_C^2$.
\end{enumerate}
\end{lem}

\begin{proof}
Since the morphism $\nu$ is finite, its image is a closed subvariety $C\subset\A^3_{k}$, which is geometrically integral because so is $\A^1_{k}$. That its ideal is exactly $I_C$ is explained in \cite[Example 3 p.\,29]{Northcott}. This proves (i).
As $f_1(t^3,t^4,t^5)=0$, we see that $f_1\in I_C$. 
To show that $f_1\notin I_{C,\langle u,v,w\rangle}^2$, notice that in the development of an element of $I_{C,\langle u,v,w\rangle}^2$ as a power series in $u$, $v$ and $w$, no term of degree $\leq 3$ may appear (this argument may be found in \cite[Example 3 p.\,29]{Northcott}).
The last assertion of (ii) follows from the identity 
\begin{equation}
\label{symboleq}
vf_1=u(v^2-uw)^2+(w^2-u^2v)(wv-u^3).\qedhere
\end{equation}
\end{proof}

Now comes our first application of Lemma \ref{lemcrit}.

\begin{thm}
\label{cxbad}
Consider the ideal $\km:=\langle x,y,z^2+1\rangle\subset\R[x,y,z] $. The polynomial 
\begin{equation*}
f:=x^{10}+x^2y^6+(z^2+1)^3-3x^4y^2(z^2+1)
\end{equation*}
is positive semidefinite. It is a sum of squares in $\R[x,y,z]_{\kp}$ for all prime ideals $\kp\subset\R[x,y,z]$ distinct from $\km$, but it is not a sum of squares in $\R[x,y,z]_{\km}$.
\end{thm}

\begin{proof}
That $f$ is positive semidefinite stems from the inequality between the arithmetic and the geometric means of  $x^{10}$, $x^2y^6$ and $(z^2+1)^3$.

Let $\psi:\A^3_{\R}\to\A^3_{\R}$ be the morphism defined by $\psi(x,y,z)=(x^2,y^2, z^2+1)$. Since the pull-back morphism $\psi^*:\R[u,v,w]\to\R[x,y,z]$ endows $\R[x,y,z]$ with a structure of free $\R[u,v,w]$-module, the morphism $\psi$ is finite flat.
Let $C$, $I_C$ and~$f_1$ be as in Lemma \ref{symb} applied with $k=\R$. Note that $f=\psi^*f_1$.
Let $\Gamma:=\psi^{-1}(C)\subset \A^3_{\R}$ be the curve defined by the ideal $$I_{\Gamma}:=\langle\psi^*I_C\rangle=\langle  x^6-y^2(z^2+1),y^4-x^2(z^2+1), (z^2+1)^2-x^4y^2\rangle\subset \R[x,y,z].$$

Remark first that~$f\in I_{\Gamma}$ by Lemma \ref{symb} (ii). 
The flatness of $\psi$ and \cite[Theorem~7.5~(ii)]{Matsumura} imply that $f\notin I_{\Gamma,\km}^2$, because $f_1\notin I_{C,\langle u,v,w\rangle}^2$ as proven in Lemma~\ref{symb}~(ii).
Since $\psi$ is finite flat and since $C$ is geometrically integral by Lemma~\ref{symb}~(i), the irreducible components of the curve $\Gamma$ surject to $C$. For $t\in\R_{>1}$, the curve $C$ is smooth at $(t^3,t^4,t^5)$ and the morphism~$\psi$ is \'etale with only real points above $(t^3,t^4,t^5)$. We deduce that all the irreducible components of $\Gamma$ contain a smooth real point. It follows that $\Gamma(\R)$ is Zariski-dense in $\Gamma$. In view of Lemma \ref{dense}, the residue fields of the generic points of $\Gamma$ are formally real.
Moreover, as the curve $\Gamma$ has no embedded point by flatness of $\psi$ (see \cite[Theorem 23.2]{Matsumura}), it is reduced.
Applying Lemma \ref{lemcrit} with $I=I_{\Gamma,\km}$, one shows that $f$ is not a sum of squares in $\R[x,y,z]_{\km}$. 

It remains to check that, if $\kp\subset\R[x,y,z]$ is a prime ideal distinct from $\km$, then~$f$ is a sum of squares in $\R[x,y,z]_{\kp}$. If $\kp$ is not maximal, this follows from Theorem~\ref{dim2}. From now on, we assume that $\kp$ is maximal, and we let $p\in\A^3_{\R}$ be the closed point associated with $\kp$.

We claim that $\Gamma$ is the Zariski closure of $\{(x_0,y_0,z_0)\in\R^3\mid f(x_0,y_0,z_0)=0\}$. We have already seen that $f$ vanishes on $\Gamma$ and that $\Gamma(\R)$ is Zariski-dense in $\Gamma$. Conversely, if $(x_0,y_0,z_0)\in\R^3$ is such that $f(x_0,y_0,z_0)=0$, we deduce from the case of equality in the inequality between the arithmetic and the geometric means that $x_0^{10}=x_0^2y_0^6=(z_0^2+1)^3$. These equations imply that $x_0\neq 0$, so $x_0^8=y_0^6$. One then easily verifies that $(x_0,y_0,z_0)$ satisfies the defining equations of $\Gamma$, which proves the claim. By \cite[Theorem~7.2.3]{BCR}, the Zariski closure of the image of $\Sper(\R[x,y,z]/\langle f\rangle)\to\Spec(\R[x,y,z]/\langle f\rangle)$ is also equal to $\Gamma$. If~$p$ does not belong to~$\Gamma$, we deduce from Theorem \ref{totpos} that~$f$ is a sum of squares in $\R[x,y,z]_{\kp}$. 

Assume from now on that $p$ belongs to $\Gamma$. In this case, we show that $f$ is a sum of squares in $\R[x,y,z]_{\kp}$ by applying Proposition \ref{glob} with $X=\A^3_{\R}$, $Y=\Gamma$, $n=3$ and $c=2$. Let us verify its hypotheses. Let $q\in\A^3_{\R}$ be the point associated with $\km$. Note that $p\neq q$ by hypothesis.
The polynomials $x^8-y^6$ and $x^{10}-(z^2+1)^3$ vanish on $\Gamma$ and have independent differentials along $\Gamma\setminus\{q\}$. We deduce that~$\Gamma$ is smooth at~$p$.
Suppose for contradiction that the Hessian of $f$ at $p$ has rank $\leq 1$.~Then 
$$\begin{pmatrix}
\frac{\partial^2 f}{\partial x\partial y} & \frac{\partial^2 f}{\partial y^2}\\
\frac{\partial^2 f}{\partial x\partial z} & \frac{\partial^2 f}{\partial y\partial z}
\end{pmatrix}=144x^5y^2z(4y^4+x^2(z^2+1))$$ vanishes at $p$. As the polynomials $x$, $y$ and $4y^4+x^2(z^2+1)$ do not vanish on $\Gamma\setminus\{q\}$, we see that $z$ vanishes at $p$. It follows that, at the point $p$, one has
\begin{equation}
\label{matrix}
\begin{pmatrix}
\frac{\partial^2 f}{\partial x^2} & \frac{\partial^2 f}{\partial x\partial y}\\
\frac{\partial^2 f}{\partial x\partial y} & \frac{\partial^2 f}{\partial y^2}
\end{pmatrix}
=\begin{pmatrix}
90x^8+2y^6-36x^2y^2&12xy^5-24x^3y\\
12xy^5-24x^3y & 30x^2y^4-6x^4
\end{pmatrix},
\end{equation}
and this quantity must vanish at $p$ by the hypothesis on the Hessian. Since $z$ vanishes at $p$, the equations of $\Gamma$ show that $x^6=y^2$ and $y^4=x^2$ at the point $p$. Combining this with (\ref{matrix}) shows that $
\begin{pmatrix}
56x^2y^2& -12x^3y\\
-12x^3y& 24x^4
\end{pmatrix}=1200x^6y^2$ vanishes at~$p$. As neither $x$ nor $y$ vanish on $\Gamma\setminus\{q\}$, this is a contradiction.
We may thus apply Proposition \ref{glob} to complete the proof of the theorem.
\end{proof}

\subsection{Sums of squares in the completion}
\label{changevar}

In \S\S\ref{changevar}-\ref{formal}, we use Lemma~\ref{lemcrit} to prove Theorem~\ref{realbad}.
To do so, we construct an example of $(A,I,f)$ as in Lemma~\ref{lemcrit}, where~$A$ is local regular with maximal ideal $\km$ and~$f$ is positive semidefinite, such that~$f$ is moreover a sum of squares in $\whA_{\km}$. If one requires the residue field of~$A$ to be formally real, this is not easy to achieve. This is the goal of Proposition~\ref{cdvariables} in \S\ref{changevar} and of Proposition \ref{lempos} in \S\ref{formal}. Let us first explain the principle of the argument of Proposition~\ref{cdvariables}, where we ensure that $f$ is a sum of squares in $\whA_{\km}$.

 Starting from the example of $(A,I,f)$ with $f\in I^{(2)}\setminus I^2$ given by Lemma~\ref{symb}, we add to $f$ a lot of squares of elements of $I$ so as to improve the chances that it is a sum of squares in $\whA_{\km}$.
 This works well only if the multiplicities of the squares of the generators of $I$ are low compared to the multiplicity of $f$, and we can only arrange this after a change of variables of relatively high degree. Making sure that $\Sper(A/I)$ remains Zariski\nobreakdash-dense in $\Spec(A/I)$ only complicates the change of variables that we need to use. The verification that the resulting element $f$ is indeed a sum of squares in $\whA_{\km}$ is computational since we do not know of a conceptual way to check it. 

We recall that, in the whole of Section \ref{sec2}, we have fixed a field $k$ of characteristic~$0$.

\begin{prop}
\label{cdvariables}
There exist ${f_2\in k[x,y,z]}$ and an ideal $I_D\subset k[x,y,z]$ such that:
\begin{enumerate}[(i)]
\item The ideal $I_D$ defines a geometrically integral curve $D\subset\A^3_{k}$. 
\item The point $(0,0,0)$ belongs to $D$. The curve $D\setminus\{(0,0,0)\}$ has a smooth $k$-point.
\item One has $f_2\in I_D$ and $f_2\notin I_{D,\langle x,y,z\rangle}^2$. 
\item There exists $h\in k[x,y,z]$ such that $h\notin I_D$ and $hf_2\in I_D^2$. 
\item The polynomial $f_2$ is a sum of squares in $k[[x,y,z]]$.
\end{enumerate}
\end{prop}

\begin{proof}
Let $\phi:\A^3_{k}\to\A^3_{k}$ be defined by $\phi(x,y,z)=(x^2,y^8-y^{10}+y^{11},-z^2+2z^3)$. Since the pull-back morphism $\phi^*:k[u,v,w]\to k[x,y,z]$ endows $k[x,y,z]$ with a structure of free $k[u,v,w]$-module, the morphism $\phi$ is finite flat.

Let $I_C$, $C$ and $f_1$ be as in Lemma \ref{symb}.
Let $D:=\phi^{-1}(C)\subset \A^3_{k}$ be defined by the ideal $I_D:=\langle\phi^*I_C\rangle\subset k[x,y,z]$. 
Let $\ok$ be an algebraic closure of $k$.
Since $\phi$ is flat, the curve $D$ has no embedded point (see \cite[Theorem 23.2]{Matsumura}). 
To prove~(i), it thus suffices to show that $D_{\ok}$ is irreducible and generically reduced, \ie that its total ring of fractions 
$$F:=\ok(t)[x,y,z]/\langle x^2-t^3, y^{11}-y^{10}+y^8-t^4,2z^3-z^2-t^5\rangle$$
is a field. Remark that $F=\ok(s)[y,z]/\langle y^{11}-y^{10}+y^8-s^8,2z^3-z^2-s^{10}\rangle$, where $s:=x/t$. Since $2z^3-z^2$ is not a nontrivial power in $\ok(z)$, we see that $2z^3-z^2-s^{10}$ is irreducible in $\ok(z)[s]$,  hence in $\ok(s)[z]$ by Gauss's lemma.
%In fact two successive applications of this lemma.
The same reasoning shows that $y^{11}-y^{10}+y^8-s^8$ is irreducible in $\ok(s)[y]$. The two field extensions $K:=\ok(s)[z]/\langle 2z^3-z^2-s^{10}\rangle$ and $L:=\ok(s)[y]/\langle y^{11}-y^{10}+y^8-s^8\rangle$ of $\ok(s)$ have coprime degree. Their tensor product $F=K\otimes_{\ok(s)}L$ is thus a field. This proves~(i).

One checks that $(0,0,0)$ and $(1,1,1)$ belong to $D(k)$. Since $\phi(1,1,1)=(1,1,1)$ is a smooth point of $C$, and since $\phi$ is \'etale at $(1,1,1)$, we see that $(1,1,1)$ is a smooth $k$-point of $D$. We have checked (ii).

Let $\hy\in k[[y]]$ be the element such that
$\hy^8=y^8-y^{10}+y^{11}$ and $\hy-y\in\langle y^2\rangle$. Similarly, let $\hz\in k[[z]]$ be such that $\hz^2=z^2-2z^3$ and $\hz-z\in\langle z^2\rangle$. For esthetic purposes, we also set $\hx:=x$. With this notation,
one can write
\begin{equation}
\label{genJ}
I_D=\langle \hx^6+\hy^8\hz^2, \hy^{16}+\hx^2\hz^2, \hz^4-\hx^4\hy^8\rangle,
\end{equation}
where the generators indeed belong to $k[x,y,z]$ since they only depend on $\hy$ and $\hz$ through $\hy^8$ and $\hz^2$.
For the same reason, the element defined as
\begin{equation}
\label{f2}
f_2:=-y^6\phi^*f_1+2(\hx^6+\hy^8\hz^2)^2+y^4(\hy^{16}+\hx^2\hz^2)^2+(\hz^4-\hx^4\hy^8)^2,
\end{equation}
belongs to $k[x,y,z]$ (we note that $\phi^*f_1=\hx^{10}+\hx^2\hy^{24}+\hz^6-3\hx^4\hy^8\hz^2$). 

To see that $f_2\in I_D$, combine Lemma~\ref{symb}~(ii) and (\ref{genJ}). 
 Assume for contradiction that $f_2\in I^2_{D,\langle x,y,z\rangle}$. Then, in view of (\ref{genJ}), one has $y^6\phi^*f_1\in I^2_{D,\langle x,y,z\rangle}$. This is absurd, because the monomial $\hy^6\hz^6$ appears in the development of $y^6\phi^*f_1$ as a power series in $\hx$, $\hy$ and $\hz$, but not in the development of any element of $I^2_{D,\langle x,y,z\rangle}$ as a power series in $\hx$, $\hy$ and $\hz$ (as~(\ref{genJ}) shows). This proves (iii).

Choose ${h:=y^8-y^{10}+y^{11}}$. 
As $h(1,1,1)=1\neq 0$, we see that $h\notin I_D$. But $h\phi^*f_1=\phi^*(vf_1)\in I_D^2$ by Lemma~\ref{symb}~(ii), so that $hf_2\in I_D^2$
in view of (\ref{genJ}). We have verified assertion~(iv).

To prove assertion (v), we use a decomposition $f_2=g+g'+g''$ in the ring $k[[x,y,z]]=k[[\hx,\hy,\hz]]$. We choose 
$$g:=-\hy^6\phi^*f_1+(\hx^6+\hy^8\hz^2)^2+\hy^4(\hy^{16}+\hx^2\hz^2)^2,$$
which is a sum of squares in $k[[\hx,\hy,\hz]]$ in view of the identity:
$$g=(\hx^2-\hy^6)^2(\hx^8+\hx^6\hy^6+\hx^4\hy^{12}+\hx^2\hy^{18}+\hy^{24}
+2\hx^2\hy^8\hz^2)+\hz^2\hy^4(\hx^4+\hy^2\hz^2)(\hy^{10}+\hz^2). $$
We also set
$$g':=(y^4-\hy^4)(\hy^{16}+\hx^2\hz^2)^2,$$
which is a sum of squares 
in $k[[\hx,\hy,\hz]]$, 
because $y^4-\hy^4-\hy^6/4$ is a square in $k[[\hy]]$ as its lowest degree term is $\hy^6/4$.
We finally define:
\begin{equation}
\label{h3}
g'':=(\hy^6-y^6)\phi^*f_1+(\hx^6+\hy^8\hz^2)^2+(\hz^4-\hx^4\hy^8)^2.
\end{equation}
To see that $g''$ is a sum of squares in $k[[\hx,\hy,\hz]]$, we note that $\hy^6-y^6=-\alpha\hy^8$ for some $\alpha\in k[[\hy]]$ whose constant term is equal to $3/4$. 
Pulling back equation (\ref{symboleq}) by the morphism $\phi$ and combining it with (\ref{h3}) yields the identity 
$$g''=\alpha\big(\hx^2(\hy^{16}+\hx^2\hz^2)^2-(\hz^4-\hx^4\hy^8)(\hx^6+\hy^8\hz^2)\big)+(\hx^6+\hy^8\hz^2)^2
+(\hz^4-\hx^4\hy^8)^2,$$
which we rewrite as
$$g''=\alpha\hx^2(\hy^{16}+\hx^2\hz^2)^2+(\hx^6+\hy^8\hz^2-\alpha/2(\hz^4-\hx^4\hy^8))^2
+(1-\alpha^2/4)(\hz^4-\hx^4\hy^8)^2.$$
In the latter expression, all terms are sums of squares in $k[[\hx,\hy,\hz]]$. Indeed, the power series $\alpha$ and $(1-\alpha^2/4)$ are sums of squares in $k[[\hy]]$ since their constant terms $3/4$ and $55/64$ are sums of squares in $\Q$, hence in $k$.
\end{proof}

\begin{rem}
\label{remf2}
To obtain a (complicated) closed formula for $f_2$, replace  $\phi^*f_1$ by its value $\hx^{10}+\hx^2\hy^{24}+\hz^6-3\hx^4\hy^8\hz^2$ in the formula (\ref{f2}), and use the change of variables $\hx=x$, $\hy^8=y^8-y^{10}+y^{11}$ and $\hz^2=z^2-2z^3$.
\end{rem}

%Montrer Theorem 0.3 direct ici en rendant le polynôme positif ? Semble difficile.

\subsection{Bad points cannot be tested formally}
\label{formal}

In Proposition \ref{lempos}, we modify the polynomial constructed in Proposition \ref{cdvariables} so as to make it positive semidefinite. We argue geometrically, on a well-chosen affine birational model of $\A^3_k$.

%Is the following true?
%A anneau, I idéal, \xi\in Sper(A) n'ayant pas de spécialisations dans A/I, avec clôture réelle R, g\in R. Existe-t-il f\in I avec f>_{\xi}g ? Déjà pour I=1?
%On pourrait espérer déduire un énoncé du genre : 
%Let $A$ be an (excellent regular?) ring, let $I\subset A$ be an ideal, and choose $f\in I^{(2)}$. Assume that $f$ is positive semidefinite in a neighbourhood of all points $\xi\in\Sper(A)$ at which $A/I$ is not regular. Then there exist $g_1,\dots,g_m\in I$ such that $f+\sum_{i=1}^m g_i^2\in A$ is positive semidefinite.
%qui abstrairait la Proposition ci-dessous.

\begin{prop}
\label{lempos}
There exist $f_3\in k[x,y,z]$ and an ideal $I_D\subset k[x_,y,z]$ such that:
\begin{enumerate}[(i)]
\item The ideal $I_D$ defines a geometrically integral curve $D\subset\A^3_{k}$. 
\item One has $(0,0,0)\in D(k)$. The curve $D\setminus\{(0,0,0)\}$ has a smooth $k$-point.
\item The polynomial $f_3$ is positive semidefinite and totally positive on $\A^3_k\setminus D$.
\item One has $f_3\in I_D$ and $f_3\notin I_{D,\langle x,y,z\rangle}^2$.
\item The polynomial $f_3$ is a sum of squares in $k[[x,y,z]]$.
\item The polynomial $f_3$ is a sum of squares in~$\cO_{\A^{3}_{k},p}$ for all $p\in\A^3_k\setminus\{(0,0,0)\}$.
\end{enumerate}
\end{prop}

\begin{proof} We may assume that $k=\Q$ since the general case follows by extending the scalars (use \cite[Theorem 7.5 (ii)]{Matsumura} to check that the second part of~(iv) remains valid).
Let $f_2$, $I_D$ and $D$ be as in Proposition \ref{cdvariables} applied with $k=\Q$. 
Define ${o:=(0,0,0)\in D(\Q)}$.
Assertions (i) and (ii) are exactly Proposition \ref{cdvariables} (i) and (ii). We fix a smooth $\Q$-point $q$ of $D\setminus \{o\}$.

Let $\oD$ be the closure of $D$ in~$\P^3_{\Q}$. 
Resolving the singularities of $\overline{D}\setminus\{o\}$ as in \cite[Chapter~8, Proposition~1.26]{Liu} shows the existence of a composition of blow-ups at closed points $\widetilde{\P}^3\to\P^3_{\Q}$ that is an isomorphism above $\A^3_{\Q}$ and such that $o$ is the only singular point of the strict transform $\wD\subset \widetilde{\P}^3$ of $\oD$.
Choose a very ample line bundle $\cL$ on $\widetilde{\P}^3$ and a basis $(\sigma_i)$ of $H^0(\widetilde{\P}^3,\cL)$, and define $U$ to be the complement in $\widetilde{\P}^3$ of the ample divisor $\{\sum_i\sigma_i^2=0\}$. 
Then $U\subset \widetilde{\P}^3$ is an affine open subset such that $U(\R)=\widetilde{\P}^3(\R)$.
Define $Z:=\wD\cap U\subset U$ and let $I_Z\subset\cO(U)$ be the ideal of $Z$. 

Notice that $o\in Z(\Q)$.
View $f_2$ as a rational function on $U$ that is well-defined at $o\in U(\Q)$. Hence, there exists $a\in\cO(U)$ nonzero at $o$ such that $af_2\in\cO(U)$. Let $b_1,\dots, b_m\in\cO(U)$ be generators of $I_Z$, and define:
\begin{equation}
\label{ajoutdecarres}
g:=a^2f_2+\lambda^2\cdot \sum_{i=1}^mb_i^2\in\cO(U),
\end{equation}
where $\lambda\in\Q$ is to be chosen later. 

We claim that, for all $p\in U(\R)$, and for all $\lambda\in\Q$ big enough, there exists a neighbourhood $\Omega_p$ of $p$ in $U(\R)$ such that $g$ is nonnegative on $\Omega_p$ .
We distinguish three cases. If $p\notin Z(\R)$, then $g$ is nonnegative
at~$p$ for $\lambda\gg 0$ since one of the~$b_i$ does not vanish at $p$. If $p=o$, then $f_2$ is nonnegative in a neighbourhood of $p\in U(\R)$ by Proposition \ref{cdvariables} (v) and Lemma~\ref{ruiz}, so that any $\lambda\geq 0$ works. If $p\in Z(\R)$ is distinct from $o$, then it is a smooth point of $Z_{\R}$. Consequently, after maybe permuting the $b_i$, we may assume that there exists $b'\in \cO(U)$ such that $(b_1,b_2, b')$ forms a regular system of parameters in $\hcO_{U_{\R},p}\simeq\R[[b_1,b_2,b']]$ 
%Pas besoin de Cohen si corps de representants.
and such that the ideal ${J:=I_Z\cdot\hcO_{U_{\R},p}\subset\hcO_{U_{\R},p}}$ is generated by $b_1$ and $b_2$.
 Proposition~\ref{cdvariables}~(iv) implies that~$f_2$, hence also $g$, vanish at the generic point of the spectrum of the ring $\hcO_{U_{\R},p}/J^2$, hence vanish in $\hcO_{U_{\R},p}/J^2$ by \cite[Appendix 6, Lemma 5]{ZS2} (or by \cite[(2.1)]{Hochster}).
%See also Hochster.
 As a consequence,  there exist ${\alpha,\beta,\gamma\in \R[[b_1,b_2,b']]}$ such that $g=\alpha b_1^2+\beta b_1b_2+\gamma b_2^2$ in $\R[[b_1,b_2,b']]$. If $\lambda\gg 0$, the constant terms of both~$\alpha$ and $\gamma-\beta^2/(4\alpha)$ are positive, so that there exist ${\delta,\varepsilon\in \R[[b_1,b_2,b']]}$ such that ${\delta^2=\alpha}$ and $\varepsilon^2=\gamma-\beta^2/(4\alpha)$. We may then write
${g=(\delta b_1+\beta b_2/(2\delta))^2+(\varepsilon b_2)^2}$ in $\R[[b_1,b_2,b']]$. Lemma \ref{ruiz} thus shows that $g$ is nonnegative in a neighbourhood~$\Omega_p$ of $p$ in $U(\R)$.
The claim is proved.

Since $U(\R)=\widetilde{\P}^3(\R)$ is compact, it is covered by finitely many of the $\Omega_p$. Consequently, for $\lambda\gg 0$, the function $g$ is nonnegative on $U(\R)$.
 We fix such a $\lambda$. In view of \cite[Theorem 7.2.3]{BCR}, the element $g\in \cO(U_{\R})$ is positive semidefinite. 
As the only field ordering of $\Q$ extends to $\R$, we deduce that $g\in\cO(U)$ is positive semidefinite.

View (\ref{ajoutdecarres}) as an identity in $\cO_{U,o}=\Q[x,y,z]_{\langle x,y,z\rangle}$. Choose ${a'\in \Q[x,y,z]}$ that does not vanish at $o$ such that $a'a$ and the $a'b_i$ all belong to $\Q[x,y,z]$. 
Let $b'_1,\dots,b'_{m'}\in \Q[x,y,z]$ be generators of $I_{D}$. 
Define 
\begin{equation}
\label{ajoutdecarres2}
f_3:=(a')^2g+\sum_{i=1}^{m'}(b_i')^2\in\Q[x,y,z].
\end{equation}
Since $g\in\cO(U)$ is positive semidefinite, we see that $(a')^2g$ is positive semidefinite as an element of $\Q(x,y,z)$, hence as an element of $\Q[x,y,z]$ by Lemma \ref{dense}. Assertion~(iii) follows at once from (\ref{ajoutdecarres2}).

Assertions (iv) and (v) are consequences of Proposition \ref{cdvariables} (iii) and (v)
and of the formulas (\ref{ajoutdecarres}) and (\ref{ajoutdecarres2}) since $a$ and $a'$ do not vanish at $o$.
 
 If $p\in\A^3_{\Q}$ is not a closed point, then $f_3$ is a sum of squares in~$\cO_{\A^{3}_{\Q},p}$ by Theorem~\ref{dim2} and (iii). Let us check that $f_3$ is a sum of squares in~$\cO_{\A^{3}_{\Q},q}$. To do so, we apply Proposition \ref{glob} with $X=\A^3_{\R}$, $Y=D$, $n=3$ and~$c=2$. Let us verify its hypotheses. As $q$ is a smooth $\Q$-point of $D$, the function field of $D$ is formally real by Lemma \ref{dense}. Since $f$ vanishes on $D$ by (iv) and is totally positive on $\A^3_{\Q}\setminus D$ by~(iii), we see that $D$ is the Zariski closure of $\Sper(\Q[x,y,z]/\langle f\rangle)$. As~$(a')^2g$ is positive semidefinite and vanishes at $q$, its differential at $q$ vanishes and its Hessian at $q$ is positive semidefinite (see Lemma \ref{lterm}). Hence, by (\ref{ajoutdecarres2}) and smoothness of~$D$ at $q$, the Hessian of $f_3$ at $q$ is positive semidefinite of rank~$\geq 2$.

By the above, there are only finitely many closed points $p_1,\dots,p_r$ in $\A^3_{\Q}\setminus\{o\}$, all distinct from $q$, such that $f_3$ is not a sum of squares in~$\cO_{\A^{3}_{\Q},p}$.
  By Lemma \ref{biraffine} below, there exists a birational morphism $\pi:\A^3_{\Q}\to\A^3_{\Q}$ such that $o$ and $q$ are in the open subset over which $\pi$ is an isomorphism, and such that none of the $p_i$ are in the image of $\pi$. After a change of coordinates, we may assume that $\pi(o)=(o)$. After replacing $f_3$, $q$ and~$D$ with $\pi^*f_3$, $\pi^{-1}(q)$ and the strict transform of $D$ by $\pi$, properties (i), (ii), (iii), (iv) and (v) are still satisfied, and (vi) now holds.
 \end{proof}

\begin{lem}
\label{biraffine}
Fix $n\geq 2$. Let $p_1,\dots,p_r,q_1,\dots,q_s\in\A^n_k$ be distinct closed points. Then there exists a birational morphism $\pi:\A^n_k\to\A^n_k$ such that the $p_i$ are not in the image of $\pi$ and the $q_j$
are in the open subset above which $\pi$ is an isomorphism.
 \end{lem}

\begin{proof}
Arguing by induction on $r$, we may assume that $r=1$. After a general linear change of coordinates, we may assume that the first coordinate of $p_1$ is distinct from the first coordinates of each of the $q_i$, and that the $n$-th coordinate of $p_1$ is nonzero. Let~$P$ be the minimal polynomial of the first coordinate of $p_1$. Then one may define~$\pi$ by setting 
$\pi(x_1,\dots,x_n)=(x_1,\dots,x_{n-1},P(x_1)x_n)$.
\end{proof}

We may now give our second application of Lemma \ref{lemcrit}.

\begin{thm}
\label{realbad}
There exists a positive semidefinite polynomial $f\in\R[x,y,z]$ that is a sum of squares in $\R[[x,y,z]]$ and in $\R[x,y,z]_{\kp}$ for all prime ideals $\kp\subset\R[x,y,z]$ distinct from $\langle x,y,z\rangle$, but that is not a sum of squares in $\R[x,y,z]_{\langle x,y,z\rangle}$.
\end{thm}

\begin{proof}
Let $f_3$, $I_D$ and $D$ be as in Proposition \ref{lempos} applied with $k=\R$. Set $f:=f_3$. In view of Proposition~\ref{lempos} (iii), (v) and (vi), we only need to show that $f$ is not a sum of squares in $\R[x,y,z]_{\langle x,y,z\rangle}$.

Proposition \ref{lempos} (i) and (ii) and Lemma \ref{dense} imply that the function field of~$D$ is formally real. In view of Proposition \ref{lempos} (i) and~(iv), one may apply Lemma~\ref{lemcrit} with $I=I_D$ to show that $f$ is not a sum of squares in $\R[x,y,z]_{\langle x,y,z\rangle}$. 
\end{proof}

\subsection{Bad points on varieties}
\label{general}

 We extend the example of Proposition \ref{lempos} first to higher dimensions in Proposition~\ref{lempos2}, then to arbitrary varieties in  Theorem \ref{main}.

\begin{prop}
\label{lempos2}
For all $n\geq c\geq 3$, there exist $f_4\in k[x_1,\dots,x_n]$ and an ideal ${I_Z\subset k[x_1,\dots,x_n]}$ such that, setting $W:=\{(0,\dots,0)\}\times\A^{n-c}_{k}\subset\A^n_k$, the following assertions hold:
\begin{enumerate}[(i)]
\item The variety $Z\subset\A^n_k$ defined by $I_Z$ is geometrically integral of dimension $n-c+1$ and contains $W$.
The variety  $Z\setminus W$ has a smooth $k$-point.
\item Let $\eta$ be the generic point of $W$.
One has $f_4\in I_Z$ and $f_4\notin I_{Z,\eta}^2$.
\item The polynomial $f_4$ is positive semidefinite and totally positive on $\A^n_k\setminus Z$.
\item The polynomial $f_4$ is a sum of squares in $\hcO_{\A^{n}_{k},\eta}$.
\item The polynomial $f_4$ is a sum of squares in~$\cO_{\A^{n}_{k},p}$ for all $p\in\A^n_k\setminus W$.
\end{enumerate}
\end{prop}

\begin{proof}
Let $f_3$, $I_D$ and $D$ be as in Proposition \ref{lempos}.
We consider the subvariety $Z:=\{(0,\dots,0)\}\times D\times\A^{n-c}_{k}$ of $\A^{c-3}_{k}\times\A^3_{k}\times \A^{n-c}_{k}=\A^n_{k}$, and we let $I_Z$ be the ideal of $Z$. View $f_3$ as a function on $\A^{c-3}_{k}\times\A^3_{k}\times \A^{n-c}_{k}=\A^n_{k}$ by pull-back from the second factor, and define $f_4:=f_3+\sum_{j=1}^{c-3} x_j^2$. Assertions (i), (ii), (iii) and~(iv) are consequences of Proposition \ref{lempos}. Assertion (v) follows from Proposition~\ref{lempos}~(vi) if $p\notin \A^{c-3}_{k}\times \{(0,\dots,0)\}\times \A^{n-c}_{k}$ and from (iii) and Theorem \ref{totpos} if $p\notin Z$.
\end{proof}

Now comes the main theorem of this section.

\begin{thm}
\label{main}
Let $X$ be an affine variety over $k$ and let $x\in X$ be a regular point. Define $A:=\cO_{X,x}$, with maximal ideal $\km$.
Assume that $\dim(A)\geq 3$ and that $\Frac(A)$ is formally real.
Then there exists $f\in\cO(X)$ such that:
\begin{enumerate}[(i)]
\item The regular function $f$ is a sum of squares in $\whA_{\km}$.
\item For all prime ideals $\kp\neq \km$ of $A$, the function $f$ is a sum of squares in $A_{\kp}$.
\item But $f$ is not  a sum of squares in $A$.
\end{enumerate}
\end{thm}

\begin{proof}
At any point of the proof, we may replace $X$ by an affine open neighbourhood $V\subset X$ of $x$. 
To see it, suppose that $f\in \cO(V)$ satisfies (i), (ii) and (iii). Choose $f'\in\cO(X)$ that does not vanish at $x$ with the property that $f'f\in\cO(V)$ lifts to an element $f''\in\cO(X)$. Then $f'f''\in\cO(X)$ also satisfies (i), (ii) and (iii) since $(f'f'')|_V=(f'|_V)^2f$.
As a consequence, we may assume~$X$ to be smooth and irreducible. Replacing $k$ with its algebraic closure in $k(X)$, we may assume that $X$ is geometrically irreducible. We set $n:=\dim(X)$ and $c:=\dim(A)$.

Let $f_4$, $Z$ and $W$ be as in Proposition \ref{lempos2} and let $q\in (Z\setminus W)(k)$ be a smooth $k$\nobreakdash-point (see Proposition \ref{lempos2} (i)). Let $\oX$ be a smooth projective compactification of $X$, let $Y\subset\oX$ be the closed integral subvariety whose generic point is $x$, and let $\oZ$ and~$\oW$ be the closures of $Z$ and $W$ in $\P^n_k$.  
Choose homogeneous coordinates $[y_1:\dots:y_{n+1}]$ of $\P^n_k$ with $\oW=\{y_1=\dots=y_c=0\}$ and $q=\{y_2=\dots=y_{n+1}=0\}$.

By the Artin-Lang homomorphism theorem \cite[Theorem 4.1.2]{BCR} applied over the real closure of $k$ associated with the restriction of an ordering of $\Frac(A)$, one may choose a closed point $p\in X\setminus (Y\cap X)$ whose residue field is formally real. 

Let $\cL$ be a very ample line bundle on $\oX$. Choose $e\gg0$, and let $\sigma_1,\dots,\sigma_{n+1}$ be sections in $H^0(\oX,\cL^{\otimes e})$ such that $\sigma_1,\dots, \sigma_{c}$ vanish on $Y$, such that $\sigma_2,\dots,\sigma_{n+1}$ vanish on $p$, and that are general among the sections satisfying these properties.

\begin{lem}
\label{Bertini}
 The following holds:
 \begin{enumerate}[(a)]
 \item The formula $t\mapsto [\sigma_1(t):\dots:\sigma_{n+1}(t)]$ defines a morphism ${\sigma:\oX\to \P^n_k}$. 
\item  The morphism $\sigma$ is finite flat, and \'etale at $p\in\oX$.
\item One has $\sigma(Y)=\oW$ and $\sigma(p)=q$. The point $p$ is a smooth point of $\sigma^{-1}(\oZ)$.
 \item  The subvariety $\sigma^{-1}(\oZ)\subset\oX$ is geometrically integral.
 \end{enumerate}
\end{lem}

\begin{proof}
(a) Let $\cI_{Y}$, $\cI_{\{p\}}$ and $\cI_{Y\cup\{p\}}$ be the ideal sheaves of $Y$, of $\{p\}$ and of $Y\cup\{p\}$ in~$\oX$. Since $e\gg 0$, the sheaves 
$\cI_{Y}\otimes\cL^{\otimes e}$, $\cI_{\{p\}}\otimes\cL^{\otimes e}$ and $\cI_{Y\cup\{p\}}\otimes\cL^{\otimes e}$ are all globally generated. It follows that $\{\sigma_1=0\}$ does not contain $p$. It then also follows, by induction on $1\leq i\leq n+1$, that $\sigma_i$ does not vanish identically on any irreducible component of $\{\sigma_1=\dots=\sigma_{i-1}=0\}$, and hence that
$\{\sigma_1=\dots=\sigma_i=0\}$ has dimension $n-i$. When $i=n+1$, this means that the $\sigma_i$ have no common zero. 

(b) Each fiber of $\sigma$ has the property that one of the $\sigma_i$ does not vanish at all on it. Since $\cL$ is ample and $\oX$ is proper, this shows that no fiber of $\sigma$ may be positive-dimensional. The morphism $\sigma$ is thus quasi-finite, hence finite since it is proper. That $\sigma$ is flat now follows from \cite[Theorem 23.1]{Matsumura}. As the sheaves $\cI_{\{p\}}\otimes\cL^{\otimes e}$ and $\cI_{Y\cup\{p\}}\otimes\cL^{\otimes e}$ are globally generated and the $\sigma_i$ are general, the differentials of $\sigma_2,\dots,\sigma_{n+1}$ at $p$ are linearly independent. The fiber $\{\sigma_2=\dots=\sigma_{n+1}=0\}$ of~$\sigma$ through $p$ is thus smooth of dimension $0$ at $p$. This completes the verification that~$\sigma$ is \'etale at~$p$.

(c) The inclusion $\sigma(Y)\subset \oW$ holds by our choice of~$\sigma_1,\dots,\sigma_c$. Since $\sigma$ is finite by (b), a dimension argument shows that $\sigma(Y)=\oW$. Our choice of $\sigma_2,\dots,\sigma_{n+1}$ implies that $\sigma(p)=q$. Since $q$ is a smooth point of $\oZ$, we deduce from (b) that $p$ is a smooth point of $\sigma^{-1}(\oZ)$. 

(d) Assertion (d) is a consequence of Bertini's irreducibility theorem. In what follows, we explain how to reduce it to the classical statement \cite[Th\'eor\`eme~6.3~4)]{Jouanolou}.

The subvariety $\sigma^{-1}(\oZ)$ of $\oX$ has no embedded point by \cite[Theorem~23.2]{Matsumura} which applies by flatness of $\sigma$, and has $p$ as a smooth closed point by (c). To prove~(d), it thus suffices to verify that $\sigma^{-1}(\oZ)$ is geometrically irreducible.
Define ${\Omega:=\{\sigma_{1}\neq 0\}\subset\oX}$.
By finiteness of $\sigma$, none of the irreducible components of $\sigma^{-1}(\oZ)$ lie over the hyperplane ${\{y_{1}=0\}}$,
so we only need to show that $\sigma^{-1}(\oZ)\cap\Omega$ is geometrically irreducible. 

Consider the open subset $\Theta:=\{y_{1}\neq 0\}\subset \oZ$. 
For $2\leq i\leq n+1$, define $z_i:=y_i/y_{1}\in\cO(\Theta)$ and $g_i:=\sigma_i/\sigma_{1}\in\cO(\Omega)$.  The variety 
$\sigma^{-1}(\oZ)\cap\Omega$ may be naturally identified with the zero locus in $\Omega\times\Theta$ of the~$n$ equations $(z_i-g_i)_{2\leq i\leq n+1}$. 
Define $\Sigma_i:=\{z_2-g_2=\dots=z_i-g_i=0\}\subset \Omega\times\Theta$, so that $\sigma^{-1}(\oZ)\cap\Omega=\Sigma_{n+1}$.
As $\cI_{\{p\}}\otimes\cL^{\otimes e}$ and $\cI_{Y\cup\{p\}}\otimes\cL^{\otimes e}$ are globally generated and the $\sigma_i$ are general, the differentials of the $g_i$ at $p$ are general. It follows that $\Sigma_i$ is smooth of dimension $2n-c-i+2$ at $p$ and that the differential at $p$ of the first projection $\pi_i:\Sigma_i\to\Omega$ has maximal rank.

We will prove by induction on $1\leq i\leq n+1$ that 
$\Sigma_i$ is geometrically irreducible. Assertion (d) will follow by taking $i=n+1$.
 In view of Proposition \ref{lempos2} (i), both~$X$ and~$\oZ$ are geometrically irreducible,
hence so is $\Sigma_1=\Omega\times\Theta$.
 This shows that the base case of the induction is valid.

 As for the induction step, assume that $\Sigma_{i-1}$ is geometrically irreducible.
Let $(\tau^{(i)}_j)_{1\leq j\leq m_i}$ be a basis of $H^0(\oX,\cI_{Y\cup\{p\}}\otimes\cL^{\otimes e})$ if $2\leq i\leq c$ (\resp a basis of $H^0(\oX,\cI_{\{p\}}\otimes\cL^{\otimes e})$ if $c+1\leq i\leq n+1$). Set $h^{(i)}_j:=\tau^{(i)}_j/\sigma_1\in\cO(\Omega)$.
Consider the morphism $\rho_i:\Sigma_{i-1}\to\A^{m_i+1}_k$ given by $t\mapsto(h^{(i)}_1(t),\dots, h^{(i)}_{m_{i}}(t),z_i(t))$.
Since~$\sigma_i$ was chosen general, the subvariety $\Sigma_i\subset\Sigma_{i-1}$ identifies with the inverse image by~$\rho_i$ of a general affine hyperplane of $\A^{m_i+1}_k$. The facts verified above that~$\Sigma_{i-1}$ is smooth of dimension $2n-c-i+3$ at $p$ and that the differential of $\pi_{i-1}$ at $p$ has maximal rank imply that the image of $\pi_{i-1}$ has dimension $\min(n, 2n-c-i+3)$. In particular, this image cannot be included in $Y\cup\{p\}$. 
Since $e\gg 0$, the linear system generated by the $\tau^{(i)}_j$ induces an embedding of $\oX\setminus (Y\cup\{p\})$.
 It follows that the transcendence degree of the subfield of $k(\Sigma_{i-1})$ generated by the~ $h_j^{(i)}$ is equal to the dimension $\min(n, 2n-c-i+3)$ of the image of $\pi_{i-1}$, hence is $\geq 2$. We deduce that the image of $\rho_i$ has dimension~$\geq 2$. Bertini's irreducibility theorem as stated in \cite[Th\'eor\`eme~6.3~4)]{Jouanolou} shows that $\Sigma_i $ is geometrically irreducible.
This concludes the induction and the proof of the lemma.
\end{proof}

We resume the proof of Theorem \ref{main}.
Let $U\subset\P^n_k$ be an open affine subset containing $q$ and the generic point of $\oW$, and such that $f_4$ is regular on $U$. Set $V:=\sigma^{-1}(U)\cap X\subset\oX$. It is an open affine subset (note that $\sigma$ is affine by Lemma~\ref{Bertini}~(b)) containing $q$ 
and the generic point $x$ of~$Y$ 
by Lemma \ref{Bertini} (c).

We now define $f:=\sigma^*(f_4|_{U})\in \cO(V)$ and check one by one the claims of Theorem~\ref{main}.
Assertions (i) and (ii) follow from Proposition~\ref{lempos2}~(iv) and (v).
To prove assertion (iii), we consider the ideal $I\subset A$ of functions vanishing on the subscheme
$\sigma^{-1}(\oZ)\times_{\oX}\Spec(A)$
of $\Spec(A)$, and we apply Lemma~\ref{lemcrit}. Let us check its hypotheses. That~ $I$ is radical stems from Lemma~\ref{Bertini}~(d). 
Since $p$ is a smooth point of $\sigma^{-1}(\oZ)$ with formally real residue field by Lemma \ref{Bertini} (c), and since $\sigma^{-1}(\oZ)$ is integral by Lemma \ref{Bertini} (d), we deduce from Lemma \ref{dense} that the function field of $\sigma^{-1}(\oZ)$ is formally real, hence that $\Sper(A/I)$ is Zariski-dense in $\Spec(A/I)$. 
That $f\in I$ follows from the first statement of Proposition \ref{lempos2}~(ii). Finally, since $\sigma$ is flat by Lemma \ref{Bertini} (b), that $f\notin I^2$ may be deduced from the second statement of Proposition~\ref{lempos2}~(ii) by applying \cite[Theorem~7.5~(ii)]{Matsumura}. Lemma~\ref{lemcrit} now applies and shows that $f$ is not a sum of squares in $A$. 
\end{proof}

\subsection{An additional example}
\label{example}

It is not straightforward to extract a concrete polynomial from the proof of Theorem \ref{realbad}.
Giving an example in $\geq 4$ variables is much easier, as the next theorem shows.

 We use a variation on Motzkin's famous polynomial \cite[p.\,217]{Motzkin}: we have only modified its coefficients to be elements of $\R[w]$ instead of real numbers.

\begin{thm}
\label{Motzkin4}
The polynomial $f=x^6+w^2y^2z^4+w^2y^4z^2+(1-w)x^2y^2z^2$ is positive semidefinite and a sum of squares in $\R[[w,x,y,z]]$, but it is not a sum of squares in $\R(w)[[x,y,z]]$, hence not in $\R[w,x,y,z]_{\langle w,x,y,z\rangle}$ either.
\end{thm}

\begin{proof}
%To see that $f$ is positive semidefinite, note that t
The inequality between the arithmetic and geometric means of $x_0^6$, $w_0^2y_0^2z_0^4$ and $w_0^2y_0^4z_0^2$ implies that $f(w_0,x_0,y_0,z_0)\geq (3w_0^{4/3}-w_0+1)x_0^2y_0^2z_0^2\geq 0$ for all $(w_0,x_0,y_0,z_0)\in\R^4$.
This shows that $f$ is positive semidefinite. The polynomial $f$ is a sum of squares in $\R[[w,x,y,z]]$ because $1-w$ is a square in this ring.

Assume for contradiction that $f$ is a sum of squares in $\R(w)[[x,y,z]]$. 
Then, for all but countably many $w_0\in\R$ the polynomial $f(w_0,x,y,z) \in \R[x,y,z]$ is a sum of squares in $\R[[x,y,z]]$. Fix such a $w_0$ with $w_0>1$ and define the polynomial $g(y,z):=f(w_0,1,y,z)\in\R[y,z]$. One can then write $g(y,z)=\sum_i h_i^2$ for some $h_i\in \R[y,z]$. Setting $y=0$, one shows that no monomial of the form $z^a$ can appear in the $h_i$. By symmetry, no monomial of the form $y^a$ can appear in the $h_i$. The identity $g(y,z)=\sum_i h_i^2$ now implies that the coefficient of $y^2z^2$ in $g$ is nonnegative, which contradicts our choice of $w_0>1$.

That $f$ is not a sum of squares in $\R[w,x,y,z]_{\langle w,x,y,z\rangle}$ follows, since $\R(w)[[x,y,z]]$ is the completion of the localization of $\R[w,x,y,z]_{\langle w,x,y,z\rangle}$ at the ideal $\langle x,y,z\rangle$. 
\end{proof}

\section{Regular local rings without bad points}

  In this last section, we construct examples of regular local rings in which all positive semidefinite elements are sums of squares.
  The regular local rings $A$ that we consider have the peculiar feature that their function field may be ordered in a unique way.
 The idea of the construction is to start with a regular local ring $B$ and with an ordering $\xi$ of $\Frac(B)$, and to choose~$A$ to be a maximal sub-$B$-algebra of the Henselization $B^{\h}$ of $B$ to which $\xi$ lifts.

\begin{thm}
\label{regex}
For all $n\geq 0$, there exists a regular local $\R$-algebra~$A$ of dimension~$n$ such that $\Sper(A)$ consists of exactly one point, which is a field ordering of $\Frac(A)$.  
\end{thm}

\begin{proof}
If $n=0$, take $A:=\R$. If $n\geq 1$, we split the proof in seven steps.
% We split the proof in several steps.

\begin{Step}
\label{step1}
Construction of the local ring $A$.
\end{Step}

Let $y\in\P^n_{\R}$ be a closed point with complex residue field, define $B:=\cO_{\P^n_{\R},y}$, let~$\km$ be the maximal ideal of $B$ and set $L:=\R(x_1,\dots,x_n)=\Frac(B)$. 
We recall the definition of the Henselization $B^{\h}$ of $B$ (see \cite[D\'efinition~18.6.5]{EGA44}). Let $(B_i)_{i\in I}$ be a set of representatives of all isomorphism classes of local essentially \'etale $B$\nobreakdash-algebras $u_i:B\to B_i$ such that $u_i$ induces an isomorphism of residue fields.
%(this last condition is automatically satisfied here as the residue field of $B$ is algebraically closed). 
Say that $i\leq i'$ if there exists a (necessarily unique) morphism of $B$-algebras $B_i\to B_{i'}$. The set $I$ is partially ordered and filtered. 
One defines $B^{\h}:=\varinjlim_{i\in I} B_i$.

Let $\alpha_1,\dots,\alpha_n$ be $n$ elements of $\R[[t]]$ that are algebraically independent over~$\R$ (see \cite[Lemma 1]{transc}). They give rise to a morphism $\alpha:\Spec(\R[[t]])\to\A^n_{\R}$. 
 Since the $\alpha_i$ are algebraically independent, the morphism $\alpha$ induces an inclusion $\alpha^*:L\hookrightarrow\R((t))$.
The field ordering of $\R((t))$ for which $t$ is a positive infinitesimal restricts, by the inclusion $\alpha^*$, to an ordering $\xi$ of~$L$.

Define $L_i:=\Frac(B_i)$.
Consider all the subsets $J\subset I$ such that:
\begin{enumerate}[(i)]
\item For all $i\in J$, the ordering $\xi$ may be extended to an ordering of $L_i$.
\item For all $i,i'\in J$, there exists $i''\in J$ with $i''\geq i$ and $i''\geq i'$. 
\end{enumerate}
Since an increasing union of such subsets again has these two properties, we may use Zorn's lemma to choose one that is maximal with respect to the inclusion. Call it $J$. It is partially ordered and filtered, and we consider the $B$-algebra $A:=\varinjlim_{i\in J}B_i$.

The arguments used in \cite[Th\'eor\`eme 18.6.6, Corollaire 18.6.10]{EGA44} to show that $B^{\h}$ is a flat local regular $B$-algebra
with maximal ideal $\km B^{\h}$ and residue field $\C$ 
show, \textit{mutatis mutandis}, that $A$ is a flat local regular $B$-algebra
with maximal ideal~$\km A$ and residue field $\C$. Its dimension is $n$ by \cite[Proposition 6.1.1]{EGA42}.

\begin{Step}
\label{step2}
Construction of an ordering $\zeta$ of $K:=\Frac(A)$.
\end{Step}

Consider the set $Z_i\subset\Sper(L_i)$ of orderings whose image in $\Sper(L)$ is $\xi$. Since $\Sper(L)$ is Hausdorff \cite[VIII, Theorem 6.3]{Lam}, its point~$\xi$ is closed. It follows from \cite[Corollary p.\,272]{Lam} that $Z_i\subset\Sper(L_i)$ is closed, hence compact by \cite[VIII, Theorem 6.3]{Lam}. Since the~$Z_i$ are nonempty for $i\in J$ by property (i) of Step \ref{step1}, the subset $Z:=\varprojlim_{i\in J}Z_i$ of $\Sper(K)=\varprojlim_{i\in J}\Sper(L_i)$ is nonempty by Tychonoff's theorem. This shows that the field $K$ is formally real. We choose a point $\zeta\in Z$. 

\begin{Step}
\label{step3}
In the remainder of the proof, we suppose that $\Sper(A)$ contains a point~$\chi$ distinct from $\zeta$, and we use this hypothesis to contradict the maximality of $J$.

In Step \ref{step3}, we show that one may assume that $\chi$ is an ordering of~$K$.
\end{Step}

The point $\chi\in\Sper(A)$ corresponds to an ordering of $\kappa:=\Frac(A/\kp)$ for some prime ideal $\kp\subset A$. Set $c:=\dim(A_{\kp})$ and let $(t_1,\dots,t_c)$ be a regular system of parameters in $A_{\kp}$. By Cohen's structure theorem \cite[Theorem~29.7]{Matsumura},
%[Cohen, On the structure and ideal theory of complete local rings, Theorem 15].
there exists an isomorphism $\whA_{\kp}\simeq \kappa[[t_1,\dots,t_c]]$, which induces inclusions $K\subset \kappa((t_1,\dots,t_c))\subset\kappa((t_1))\dots((t_c))$. Any ordering of a field $k$ extends in two ways to an ordering of $k((t))$, one for which~$t$ is a positive infinitesimal and one for which $t$ is a negative infinitesimal. If $c\geq 1$, it follows that $K$ admits at least two orderings, one for which $t_1$ is positive and one for which $t_1$ is negative. Replacing $\chi$ by one of these, we may assume that $c=0$, \ie that $\zeta$ and $\chi$ are two distinct orderings of $K$.

\begin{Step}
\label{step4}
Study of the valuations associated with the orderings $\zeta$ and $\chi$.
\end{Step}

 Let $f\in K$ be such that $f\succ_{\zeta} 0$ but $f\prec_{\chi}0$. There exists $j\in J$ such that $f\in L_j$, and we fix such an element $j$. Since $B_j$ is a local essentially \'etale $B$-algebra, there exist a projective variety $X$ over~$\R$, a closed point $x\in X$, a morphism $\pi:X\to \P^n_{\R}$ \'etale at $x$ such that $\pi(x)=y$, and an isomorphism of $B$-algebras $B_j\simeq\cO_{X,x}$. In particular, $L_j\simeq\R(X)$. 
By resolution of singularities, we may assume that $X$ is smooth over $\R$.
 After multiplying $f$ by a square, we may assume that $f\in B_j$. Let $D\subset X$ be the effective Cartier divisor obtained by taking the Zariski closure in $X$ of the subscheme $\{f=0\}\subset\Spec(B_j)$.

Let ${V_{\prec}:=\{g\in L_j \mid -r\prec g\prec r\textrm{ for some } r\in\R\}}$ be the valuation ring associated with an ordering $\prec$ of $L_j$ (see \cite[Proposition 10.1.13]{BCR}).
%\cite[Theorems 2 and 3]{Lang},
Its maximal ideal is $\km_{\prec}:=\{g\in L_j \mid -r\prec g\prec r\textrm{ for all } r\in\R_{>0}\}$ and its residue field is isomorphic to~$\R$. 
We let $v_{\prec}$ be the corresponding valuation of $L_j$ and $c_{\prec}\in X(\R)$ be its center. Since $\zeta$ restricts to $\xi$ on $L$, the restriction of $v_{\zeta}$ to $L$ is induced by the $t$-adic valuation on $\R((t))$ and the inclusion ${\alpha^*:L\hookrightarrow\R((t))}$. As $L_j$ is a finite extension of $L$, we deduce that $v_{\zeta}$ is a discrete valuation. If $c_{\zeta}\in D$, replace $X$ with its blow-up at~$c_{\zeta}$, and $D$ with its strict transform in the blow-up. This has the effect of decreasing the image by the valuation $v_{\zeta}$ of a local equation of $D$ at $c_{\zeta}$. As $v_{\zeta}$ is discrete, repeating this procedure finitely many times ensures that $c_{\zeta}\notin D$.

\begin{Step}
\label{step5}
Construction of a subset $J'\subset I$.
\end{Step}

Let $S$ be the
spectrum of the 
semilocal ring of $X$ at the points $x$, $c_{\zeta}$ and $c_{\chi}$. We note that $x$ is distinct from either $c_{\zeta}$ or $c_{\chi}$ since its residue field is $\C$ (but $c_{\zeta}$ and $c_{\chi}$ might coincide). As $\cO_X(-D)$ is invertible and as any locally free module of constant rank over the spectrum of
a semilocal ring is free, the ideal sheaf $\cO_X(-D)|_S$ is principal, generated by some $g\in H^0(S,\cO_X(-D)|_S)$. Since
$c_{\zeta}\notin D$, one has $g(c_{\zeta})\neq 0$ and we may assume, after maybe replacing $g$ with $-g$, that $g(c_{\zeta})>0$. In particular, $g\succ_{\zeta}0$. If $c_{\zeta}=c_{\chi}$ or if $g\succ_{\chi}0$, define $h=1$. If $c_{\zeta}\neq c_{\chi}$ and $g\prec_{\chi}0$, let $h\in\cO(S)^*$ be an invertible element such that $h(x)=h(c_{\zeta})=1$ and $h(c_{\chi})=-1$, so that $h\succ_{\zeta} 0$ and $h\prec_{\chi} 0$.
Then the element $e:=fh/g\in L_j$ has the property that $e\succ_{\zeta}0$ and $e\prec_{\chi}0$. Moreover, $e\in (B_j)^*$ because both $f$ and $g$ generate the invertible sheaf $\cO_X(-D)$ at the point~$x$.

Consider the ring $A'$ obtained by localizing $A[z]/\langle z^2-e\rangle$ at one of its maximal ideals. We define $J'\subset I$ to be the set of $i\in I$ such that there exists a morphism of $B$\nobreakdash-algebras $B_i\to A'$. 
%We claim that $J'$ contradicts the maximality of $J$, which will conclude the proof of the theorem.
%We now check that $J'$ satisfies the two properties (i) and (ii) that were required of $J$.

\begin{Step}
\label{step6}
The  subset $J'\subset I$ satisfies the properties (i) and (ii) of Step \ref{step1}.
\end{Step}

Since $e\prec_{\chi}0$, the element $e\in K$ is not a square, and it follows that $A'$ is integral with fraction field $K':=K[z]/\langle z^2-e\rangle$. Since $e\succ_{\zeta}0$, the element $e$ has a square root in the real closure of $K$ associated with $\zeta$. This shows that $\zeta$ may be extended to an ordering $\zeta'$ of $K'$. If $i\in J'$, the restriction of $\zeta'$ to $L_i$ is an ordering of $L_i$ that extends $\xi$. The shows (i).

Choose $i,i'\in J'$. The two morphism $B_i\to A'$ and $B_{i'}\to A'$ induce a morphism $B_i\otimes_B B_{i'}\to A'$. The localization of $B_i\otimes_B B_{i'}$ at its maximal ideal induced by the maximal ideal of $A'$ is a local essentially \'etale $B$-algebra with residue field $\C$ that admits a morphism to $A'$. It is thus of the form $B_{i''}$ for some $i''\in J'$, and the element $i''\in J'$ satisfies $i''\geq i$ and $i''\geq i'$. We have verified the property (ii).

\begin{Step}
\label{step7}
The subset $J'\subset I$ contradicts the maximality of $J$.
\end{Step}

 It is clear that $J\subset J'$ since for all $i\in J$, there exist morphisms of $B$-algebras $B_i\to A\to A'$. 

It remains to show that $J'\neq J$. Consider the ring $B_j'$ obtained by localizing $B_j[z]/\langle z^2-e\rangle$ at its maximal ideal induced by the maximal ideal of $A'$. The ring $B'_j$ is a local $B$-algebra with residue field $\C$ that is essentially \'etale because $e\in (B_j)^*$. It is therefore of the form $B_{j'}$ for some $j'\in I$. Since there exists a morphism $B'_j\to A'$ by construction, we see that $j'\in J'$. However $j'$ cannot belong to $J$. Indeed, if it were the case, there would exist a morphism of $B$-algebras $B'_j\to A$. This is impossible since $e$ is a square in $B'_j$ but $e\prec_{\chi}0$. 
\end{proof}

That the positive semidefinite elements in the regular local rings constructed in Theorem \ref{regex} are sums of squares is a straightforward application of Scheiderer's results on sums of squares in local rings.

\begin{thm}
\label{regloc}
For all $n\geq 0$, there exists a regular local $\R$-algebra~$A$ of dimension~$n$ with the following properties:
\begin{enumerate}[(i)]
\item All positive semidefinite elements of $A$ are sums of squares in $A$.
\item The field $\Frac(A)$ is formally real.
\end{enumerate}
\end{thm}

\begin{proof}
Let $A$ be the $\R$-algebra constructed in Theorem \ref{regex}. It satisfies~(ii). To verify~(i), choose  a nonzero positive semidefinite element $f\in A$. Since the only point of $\Sper(A)$ is supported on the ideal $(0)$ of $A$,
%residue field of $A$ that is formally real is $\Frac(A)$, 
the space $\Sper(A/\langle f^2\rangle)$ is empty. It follows from the real Nullstellensatz that $-1$ is a sum of squares in $A/\langle f^2\rangle$ (see \cite[Theorem 4.3.7]{BCR}). As a consequence, $f=((f+1)/2)^2-((f-1)/2)^2$ is a sum of squares in $A/\langle f^2\rangle$, hence a sum of squares in $A$ by \cite[Corollary~2.3~(b)]{Schlocal}.
\end{proof}

 \bibliographystyle{myamsalpha}
\bibliography{badpoints}

\end{document}